\pdfoutput=1
\documentclass[11pt]{article}
\usepackage{color}
\usepackage{graphicx}
\usepackage{amsmath,amsfonts,amssymb,graphics,amsthm}
\usepackage{hyperref}
\usepackage{comment}
\usepackage{tabularx}
\usepackage[protrusion=true,expansion=true]{microtype}
\usepackage{enumerate}
\usepackage{bbm}
\usepackage{upgreek}
\usepackage{mathrsfs}
\usepackage{esint}
\usepackage[margin=1in]{geometry}
\usepackage{mathtools}
\usepackage{subcaption}
\usepackage{marginnote}

\usepackage[mathscr]{eucal}

\usepackage{booktabs}

\definecolor{byzantium}{rgb}{0.44, 0.16, 0.39}

\hypersetup{
    colorlinks=false,
    linktocpage,
    }
\numberwithin{equation}{section}

\newtheorem{theorem}{Theorem}[section]

\newtheorem{corollary}[theorem]{Corollary}

\newtheorem{proposition}[theorem]{Proposition}
\newtheorem{prop}[theorem]{Proposition}
\newtheorem{lem}[theorem]{Lemma}

\newtheorem{claim}[theorem]{Claim}
\newtheorem{ques}[theorem]{Question}

\newtheorem{definition}[theorem]{Definition}
\newtheorem{defn}[theorem]{Definition}
\newtheorem{quasidefn}[theorem]{``Definition''}
\theoremstyle{remark}
\newtheorem{remark}[theorem]{Remark}

\numberwithin{equation}{section}

\def\@rst #1 #2other{#1}
\newcommand\MR[1]{\relax\ifhmode\unskip\spacefactor3000 \space\fi
  \MRhref{\expandafter\@rst #1 other}{#1}}
\newcommand{\MRhref}[2]{\href{http://www.ams.org/mathscinet-getitem?mr=#1}{MR#2}}

\def\MR#1{\href{http://www.ams.org/mathscinet-getitem?mr=#1}{MR#1}}

\newcommand{\C}{\mathbbm{C}}

\newcommand{\E}{\mathbbm{E}}

\newcommand{\Z}{\mathbbm{Z}}
\newcommand{\R}{\mathbbm{R}}

\newcommand{\aconst}{{\frac{\chi(M)}{6}}}

\newcommand{\Vol}{{\mathrm{Vol}}}
\newcommand{\Len}{{\mathrm{Len}}}

\newcommand{\eps}{\varepsilon }

\let\Re\undefined
\DeclareMathOperator{\Re}{Re}

\DeclareMathOperator{\SLE}{SLE}

\def\cS{\mathcal{P}}

\def\cC{\mathcal{C}}

\newcommand{\aryb}{\begin{eqnarray*}}
\newcommand{\arye}{\end{eqnarray*}}
\def\alb#1\ale{\begin{align*}#1\end{align*}}
\def\allb#1\alle{\begin{align}#1\end{align}}
\newcommand{\eqb}{\begin{equation}}
\newcommand{\eqe}{\end{equation}}
\newcommand{\eqbn}{\begin{equation*}}
\newcommand{\eqen}{\end{equation*}}

\newcommand{\BB}{\mathbbm}

\newcommand{\op}{\operatorname}

\newcommand{\ep}{\epsilon }
\newcommand{\rta}{ \rightarrow }

\newcommand{\wt}{\widetilde}
\newcommand{\wh}{\widehat}
\newcommand{\mcl}{\mathcal}

\newcommand{\tr}{\op{tr}}
\newcommand{\tloop}{\mathrm{loop}}

\newcommand{\qq}{Q}
\newcommand{\cc}{{\mathbf{c}}}

\DeclareMathAlphabet{\mathpzc}{OT1}{pzc}{m}{it}

\begin{document}

\author{
\begin{tabular}{c}Morris Ang\\[-5pt]\small MIT\end{tabular}\;
\begin{tabular}{c}Minjae Park\\[-5pt]\small MIT\end{tabular}\;
\begin{tabular}{c}Joshua Pfeffer\\[-5pt]\small MIT\end{tabular}\;
\begin{tabular}{c}Scott Sheffield\\[-5pt]\small MIT\end{tabular}}

\title{Brownian loops and the central charge of a Liouville random surface}

\date{}

\maketitle

\begin{abstract}
We explore the geometric meaning of the so-called {\em zeta-regularized} determinant of the Laplace-Beltrami operator on a compact surface, with or without boundary.  We relate the $(-c/2)$-th power of the determinant of the Laplacian to the appropriately regularized partition function of a Brownian loop soup of intensity $c$ on the surface.  This means that, in a certain sense, {\em decorating} a random surface by a Brownian loop soup of intensity $c$ corresponds to {\em weighting} the law of the surface by the $(-c/2)$-th power of the determinant of the Laplacian.

Next, we introduce a method of regularizing a Liouville quantum gravity (LQG) surface (with some matter
central charge parameter $\cc$) to produce a smooth surface. And we show that weighting the law
of this random surface by the $( -\cc'/ 2)$-th power of the Laplacian determinant has precisely the
effect of changing the matter central charge from $\cc$ to $\cc + \cc'$. Taken together with the earlier results, this provides a way of interpreting an LQG surface of matter central charge $\cc$ as a pure LQG surface {\em decorated} by a Brownian loop soup of intensity $\cc$.

Building on this idea, we present several open problems about random planar maps and their continuum analogs. Although the original construction of LQG is well-defined only for $\cc \leq 1$, some of the constructions and questions also make sense when $\cc>1$.
\end{abstract}

\tableofcontents

\section{Introduction} \label{sec-intro}

Liouville quantum gravity (LQG) surfaces are actively studied in the probability community not only for their applications to conformal field theory, string theory, and other areas of mathematical physics, but also for their close connection to the geometry of many classes of random planar maps.
The physics literature on this topic is vast, including many perspectives---such as Coulomb gas theory and conformal field theory---that we will not survey here. An early definition of LQG surfaces, as presented by Polyakov and others in the 1980s, can be formulated roughly as follows. (See also Proposition~\ref{prop_polyakov_alvarez} below.)
\begin{quasidefn}\label{def-original}
Let $\cc$ be some value in $(-\infty,1]$. An LQG surface with matter central charge\footnote{In the physics literature, the matter central charge is often denoted as $\cc_{\text{M}}$ to distinguish it from the Liouville central charge, which they denote as $\cc_{\text{L}}$.  We do not use this notation since we do not refer to the Liouville central charge in this paper.} $\cc$ and  a specified topology (such as the sphere, disk, or torus) is a random surface whose law is ``the uniform measure on the space of surfaces with this topology'' weighted by the $(-\cc/2)$-th power of the determinant of the Laplace-Beltrami operator of the surface.
\end{quasidefn}

The main goal of the current paper is to provide a rigorous geometric interpretation of the above, and to justify and clarify its relationship to the other mathematical definitions of LQG surfaces that have been used in recent years. We will also make connections between the classical theory of determinants of Laplace-Beltrami operators on smooth surfaces and the more recent literature about Brownian loop on those surfaces.

When $\cc$ is a positive integer, the quantity $(\det \Delta)^{-\cc/2}$ represents, at least heuristically, the partition function of a $\cc$-dimensional Gaussian free field on the surface. (We recall the definition of the Gaussian free field in Section~\ref{sec-square-subdivision}; for a more comprehensive treatment, see, e.g.,~\cite{shef-gff, berestycki2015introduction, werner2020lecture}.) This heuristic relationship corresponds to an exact identity in the discrete setting, which we describe later in this section.  The partition function $(\det \Delta)^{-\cc/2}$ heuristically\footnote{This heuristic is described more formally, e.g., in~\cite{d'hoker-phong}.  In summary, the Polyakov partition function---i.e., the partition function of LQG---is an integral of the exponential of minus the Polyakov action over all embedded manifolds in spacetime, where spacetime is assumed to be $\R^{\cc}$.  And, if one first fixes the metric and integrates the exponential of minus the Polyakov action over all embeddings of the manifold equipped with that metric, then the resulting integral is a Gaussian integral and is heuristically seen to equal the determinant of the Laplacian raised to the power $-\cc/2$.} measures the ``number of ways to embed the surface into $\mathbb R^{\cc}$.'' Informally, this means that, if one chooses a uniformly random surface {\em embedded} in $\mathbb R^{\cc}$ and then forgets the embedding, then the resulting un-embedded surface has the law of an LQG surface with matter central charge $\cc$.  Note, however, that this intuition holds only for positive integers $\cc$, while the idea of weighting by $(\det \Delta)^{-\cc/2}$ makes sense heuristically for any real value of $\cc$.

Much of the recent mathematical work on LQG surfaces rests on a completely different formulation of LQG, derived from the above definition by David~\cite{david-conformal-gauge} and Distler-Kawai~\cite{dk-qg} via the so-called {\em DDK ansatz}.  In the case of simply connected topology, this alternative formulation of LQG defines an LQG surface with matter central charge $\cc$ as a random Riemannian manifold with the specified topology whose Riemannian metric tensor is heuristically given by multiplying the constant curvature metric by the exponential of some multiple of the Gaussian free field (GFF).

We review the definition of the GFF in Section~\ref{sec-square-subdivision}.
By considering different variants of the Gaussian free field---different boundary conditions, centering the field in different ways, conditioning on the total LQG area of the domain, adding logarithmic singularities, etc.---it is possible to define many different types of LQG surfaces whose corresponding fields look {\em locally} like the Gaussian free field.

At first glance, these definitions are also heuristic because the Gaussian free field $h$ is a random distribution, and cannot be exponentiated.  But by replacing $h$ with a mollified version of the field and taking limits, it is possible to construct LQG surfaces rigorously as random surfaces equipped with a conformal structure, a volume measure~\cite{hoegh1971general, kahane} (see also~\cite{rhodes-vargas-review,berestycki-gmt-elementary,aru-gmc-survey, shef-kpz}), and a distance function~\cite{dddf-lfpp,local-metrics,lqg-metric-estimates,gm-confluence,gm-uniqueness,gm-coord-change}.\footnote{This distance function is often referred to as a metric in the literature, in the sense of metric spaces.  We avoid using the term ``metric'' in this context to avoid confusion with the notion of a Riemannian metric.}

\begin{figure}[t!]
\centering
        \includegraphics[width=.8\textwidth]{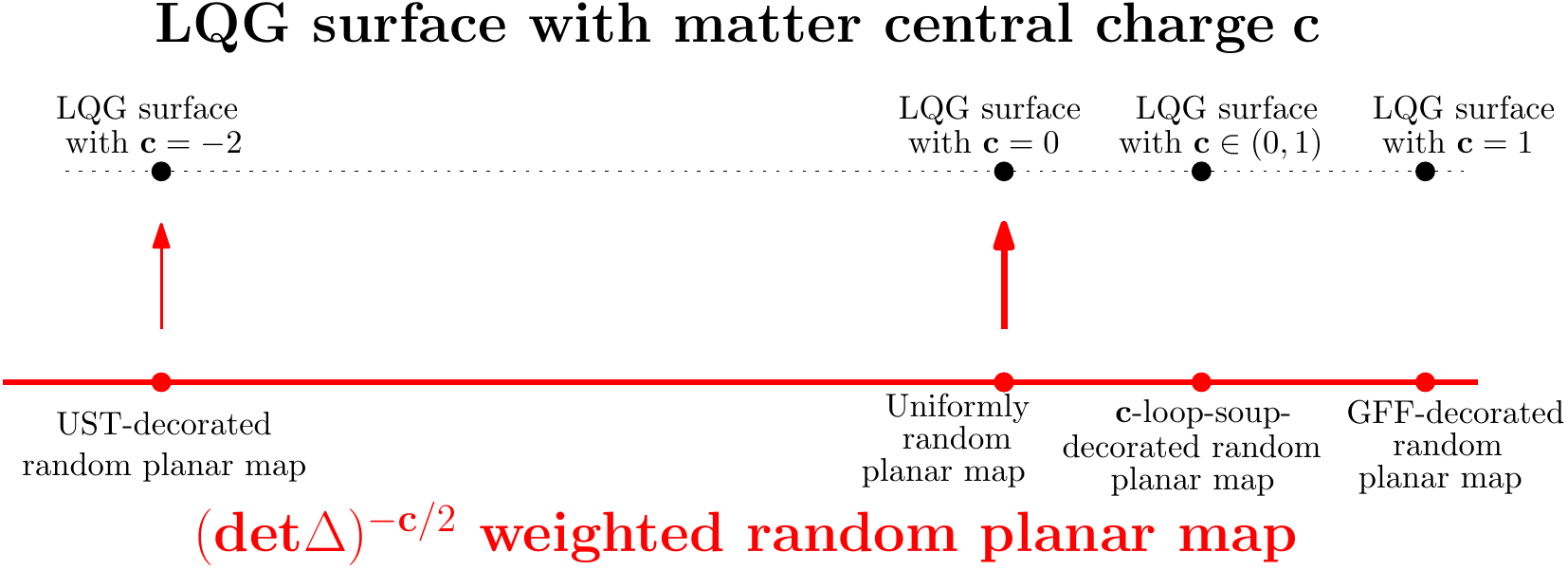}
    \caption{The uniformly random planar map with $n$ vertices {\em weighted} by the $(-\cc/2)$-th power of the Laplacian determinant, for the appropriate notion of discrete Laplacian, can be interpreted as a uniform planar map {\em decorated} by some number of spanning trees, GFF instances, or $\cc$-intensity (random walk) loop soups.  For each $\cc \leq 1$, it is believed to converge as $n \to \infty$ to an LQG surface with matter central charge $\cc$: for example, if the random planar map has the topology of the sphere, the limit should be the LQG sphere.  However, this has been proven only for two values of $\cc$.  For $\cc=0$, the convergence is known in several topologies: as a Gromov-Hausdorff limit of random metric spaces, as a weak limit of random measures induced by Cardy embeddings (in the case of uniform triangulations), and in the so-called ``mating-of-trees'' or ``peanosphere'' topology described in~\cite{wedges}.  And, for $\cc=-2$, the convergence is known in the mating-of-trees/peanosphere topology.  (The other mating-of-trees convergence results apply to models such as FK clusters, bipolar orientations, Schnyder woods, and the Ising and Potts models. Weighting by the partition functions of models such as these is {\em believed}~\cite{bpz-conformal-symmetry} to be equivalent in the $n \rta \infty$ limit to weighting by powers of the Laplacian determinant.  See \cite{wedges, ghs-mating-survey}.)} \label{RPMchart}
\end{figure}

Although the approaches discussed above give satisfactory mathematical definitions of LQG surfaces, it is not obvious how they relate to ``Definition''~\ref{def-original}. 

One approach to making sense of ``Definition''~\ref{def-original} is by considering limits of discrete models. In the special case of matter central charge $\cc = 0$, it is possible to make sense of sampling a surface uniformly according to ``uniform measure on the space of surfaces with this topology'' by considering uniform random planar maps (or uniform $p$-angulations for $q=3$ or $q$ even) with a fixed number of edges and taking a limit as the number of edges grows to infinity---either in the Gromov-Hausdorff sense to obtain a  limiting random distance function~\cite{lqg-tbm1,lqg-tbm2,lqg-tbm3,legall-uniqueness,miermont-brownian-map} or under certain ``discrete conformal embeddings into the plane'' to additionally obtain a random measure~\cite{hs-cardy-embedding}.  The limiting random distance function and measure  are equal in law to the random LQG distance function and measure for $\cc = 0$.%

We can try to extend this approach to general $\cc$ by considering classes of random planar maps sampled with probability proportional to the $(-\cc/2)$-th power of the determinant of its  \emph{discrete Laplacian}.  (We define two different notions of the discrete Laplacian in Section~\ref{sec-discrete}.)  However, this perspective does \emph{not} yield a rigorous version of ``Definition''~\ref{def-original} because we do not know how to prove that such random planar maps converge in the distance function or measure sense for general $\cc$.  This situation is summarized in Figure~\ref{RPMchart}.

In this paper, we make ``Definition''~\ref{def-original} precise in the continuum by taking the ``uniform measure'' in ``Definition''~\ref{def-original} to be the one obtained as the measure and distance function scaling limits of random planar maps in the special case $\cc=0$, and then by defining the laws for other values of $\cc$ by weighting the $\cc = 0$ law by the appropriate power of the determinant of the Laplace-Beltrami operator.  To implement this approach, we must address two main issues.  First, it is not {\em a priori} clear how to define the determinant of the Laplacian: the spectrum of the Laplacian of a smooth compact Riemannian manifold (which is known to be discrete) is unbounded, so the product of its eigenvalues is infinite; and for random fractal surfaces, the situation is still worse, as even defining these eigenvalues would take some work.  Second, the laws of the quantum gravity surfaces are mutually singular for different values of the matter central charge, so the law of a surface for one value of the matter central charge cannot, strictly speaking, be expressed as the weighted law of an LQG surface for a different value of the matter central charge.

In this paper, we overcome these difficulties by regularization. We first define a  manifold approximation of an LQG surface.  %
The most obvious regularization approach is to regularize the GFF $h$ at a ``Euclidean scale''---for instance, by replacing $h$ with its circle average, or by mollifying $h$ with a smooth bump function.  However, this turns out not to be the correct approach, as it does not commute with scaling in the natural way; i.e., the amount a portion of the surface gets ``smoothed'' depends on its size in the conformal embedding, not on its intrinsic size. The second most obvious approach---considering the surface obtained from an LQG surface after a finite time interval of Ricci flow---should be correct in spirit, but it is more complicated to construct and work with because it is not a linear function of $h$.  Instead, this paper will use a third type of regularization, which is morally a discretized version of the Ricci flow approach. We regularize the field $h$ by projecting it onto a finite-dimensional subspace with prescribed averages on dyadic squares with roughly the same size w.r.t.\ the LQG measure; see Section~\ref{sec-square-subdivision} for details and Figure~\ref{fig-square-subdivision-realization} for a simulation of this square subdivision. Essentially, we regularize the field, not at a Euclidean scale, but at a ``quantum scale''. 

We use this regularized version $h^\ep$ of $h$ to define a random manifold approximation of an LQG surface.  The heuristic LQG metric is most often expressed in the literature as ``$e^{\gamma h} (dx^2 + dy^2)$'', but because we regularized the field at a quantum rather than Euclidean scale, our regularized LQG metric contains a factor of $2/Q$ rather than $\gamma$ in the exponent.  (See~\eqref{eqn-def-gamma-Q} for definitions of both $\gamma$ and $Q$.) In Section~\ref{sec-reg-intuition}, we explain why our form for the heuristic LQG metric, with $2/Q$ instead of $\gamma$, can be viewed as more natural from a geometric perspective.

We then weight the law of our quantum-regularized LQG surface by the so-called zeta-regularized determinant of the Laplace-Beltrami operator associated to the manifold.  We define the zeta-regularized determinant of the Laplacian in Section~\ref{sec-reg-det} and interpret this quantity in Theorem~\ref{thm-loop-det} below.  We then obtain a ``regularized'' interpretation of ``Definition''~\ref{def-original}, which is also illustrated by Figure~\ref{RLQGchart}.  Loosely stated, our result is as follows:

\begin{theorem}\label{thm-weighting}
For any $\cc,\cc'$ with $\cc$ and $\cc+\cc'$ in the range $(-\infty,1)$, the law of the quantum-regularized LQG surface just described with matter central charge $\cc+\cc'$ is equal to the law of the same quantum-regularized LQG surface with matter central charge $\cc$ weighted by the $(-\cc'/2)$-th power of the zeta-regularized determinant of the Laplacian.
\end{theorem}

We state this result more precisely in Theorem~\ref{thm-laplacian-reweighting}.  In stating and interpreting this more precise formulation of our result, we will make two important observations:

\begin{enumerate}
\item
The quantum-regularized LQG surface in Theorem~\ref{thm-weighting} can be taken to be a surface that heuristically approximates a type of quantum sphere; see Claim~\ref{claim-sphere}.
\item
We have stated Theorem~\ref{thm-weighting} only for $\cc$ and $\cc + \cc'$ less than $1$, but Theorem~\ref{thm-laplacian-reweighting} applies to all values of $\cc$ and $\cc+\cc'$ less than $25$.  For this extended range of matter central charge, the LQG measure and distance function
 are no longer defined, so it is not clear \textit{a priori} how any notion of ``regularized LQG surface'' we define would relate to LQG in this regime. In Section~\ref{sec-interpretation2}, we interpret Theorem~\ref{thm-laplacian-reweighting} for matter central charge between $1$ and $25$ in the context of the recent work~\cite{ghpr-central-charge} on LQG in this phase.  This will help explain an apparent conflict between the model of LQG for matter central charge in $(1,25)$ in~\cite{ghpr-central-charge} and suggested descriptions of LQG in this phase in the physics literature.
\end{enumerate}

\begin{figure}[t!]
\begin{center}
\centering
        \includegraphics[width=.8\textwidth]{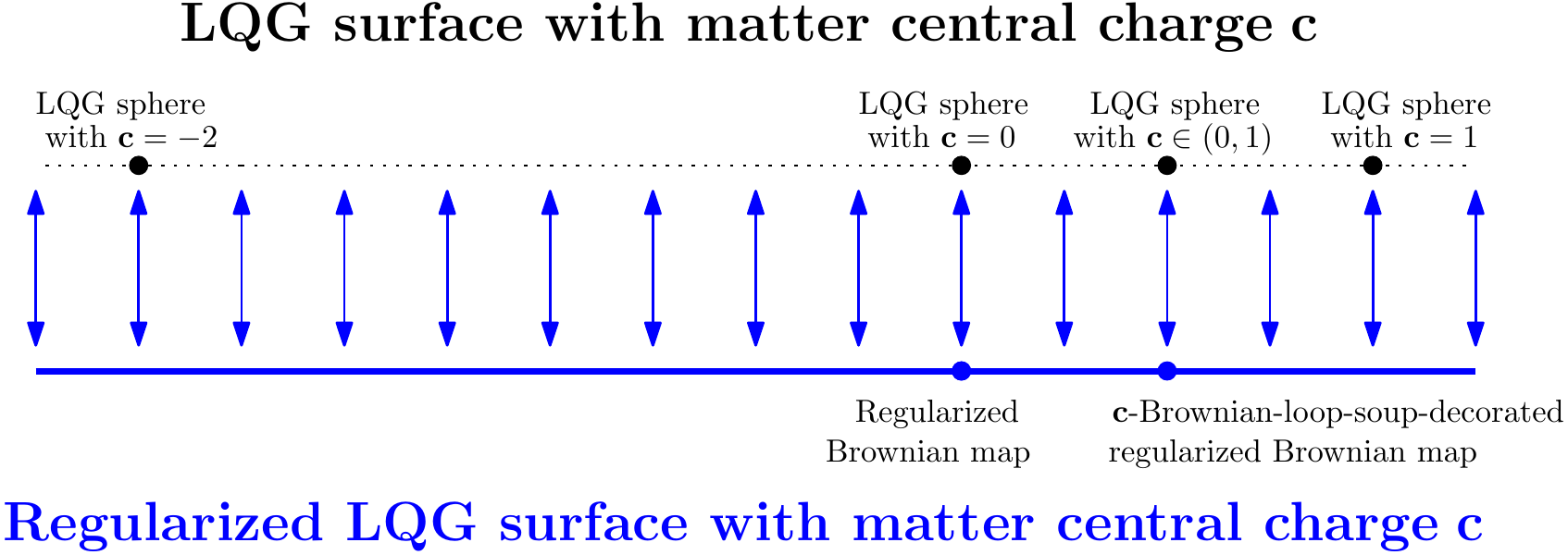}
    \caption{Theorem~\ref{thm-weighting} states that a regularized LQG surface with matter central charge $\cc$ can be obtained by starting with a regularized LQG surface with matter central charge $0$ (a.k.a.\ the regularized Brownian map) and {\em weighting} by $(\det_{\zeta}' \Delta)^{-\cc/2}$. Thus, the blue line here is analogous to the red line in Figure~\ref{RPMchart} but with ``random planar map'' replaced by ``regularized Brownian map'' and the determinant interpreted in a continuum sense. In this figure, however, the correspondence between the horizontal lines is straightforward for {\em all} $\cc < 25$. (Note that,  in the figure, we have implicitly related the zeta-regularized determinant of the Laplacian to the partition function of the Brownian loop soup; see the remark after Theorem~\ref{thm-loop-det} and the discussion at the end of Section~\ref{sec-relate}.) }  \label{RLQGchart}
\end{center}
\end{figure}

We define our manifold approximation of LQG and prove Theorem~\ref{thm-weighting} in the second half of the paper. In the first half we present and explore the definition of the zeta-regularized determinant of the Laplacian.

As we mentioned above, the determinant of the Laplacian is not strictly defined on a smooth compact Riemannian manifold, since the spectrum of the Laplacian is unbounded and therefore the product of its eigenvalues is infinite.  To circumvent this difficulty, we make sense of the determinant by the standard technique of zeta-regularization, which we describe in Section~\ref{sec-reg-det}.  At first glance, the definition of the zeta-regularized determinant might seem to be {\em ad hoc} or to be just one of many equally reasonable ways to define a Laplacian determinant on a manifold.  We show that this regularization of the determinant is indeed canonical from a geometric perspective by proving a relation between this quantity and an appropriately regularized mass of loops on the surface with respect to the \emph{Brownian loop measure}.  This relation mirrors an exact identity in the discrete setting---which we describe in Section~\ref{sec-discrete}---between the determinant of the discrete (random walk) Laplacian on a graph  and masses of loops  under discrete loop measures on the graph.

The Brownian loop measure was introduced by \cite{lawler-werner-soup} for planar domains.  And, as described in~\cite[Section 6.3]{werner-self-avoiding-loops}, their definition extends easily to the general setting of a Riemann surface $M$, possibly with boundary, equipped with a smooth metric $g$.  We review this material in Section~\ref{sec-brownian-loops}.

The mass of Brownian loops, like the Laplacian determinant, needs to be regularized, since there are infinitely many arbitrarily small Brownian loops on a surface---and, in the case of closed manifolds, infinitely many arbitrarily large Brownian loops as well.  (We define the ``size'' of a Brownian loop by its quadratic variation.) %
The most straightforward way to regularize these masses of loops is by a simple cutoff, i.e., restricting the quadratic variation to lie in a fixed interval.\footnote{Note that the regularizing intervals in the statements of our results are defined with the parameters multiplied by factors of $4$; this is just because the resulting formulas are slightly neater.  %
 } This truncated mass of loops can be made to converge as this fixed interval grow to all of $[0,\infty)$ by subtracting terms (that are blowing up as the interval grows) that depend only on the area and boundary length of the surface.  The following theorem asserts the resulting limit is given precisely by the zeta-regularized determinant of the Laplacian.  

\begin{theorem}\label{thm-loop-det}
For manifolds $(M,g)$ with boundary, the total mass of loops with quadratic variation greater than $4\delta$ under the Brownian loop measure is given by
\[\frac{\Vol_{g}(M)}{4 \pi \delta} - \frac{\Len_g(\partial M)}{4\sqrt{\pi \delta}} - \log {\det}_{\zeta} \Delta_g - \aconst (\log \delta + \upgamma) + O(\delta^{1/2}) \]
as $\delta \rta 0$, where $\upgamma \approx 0.5772$ is the Euler-Mascheroni constant.
For closed manifolds $(M,g)$, the total mass of loops with quadratic variation greater than $4\delta$ and less than $4C$ is
\[
\frac{\Vol_g(M)}{4 \pi \delta} - \aconst (\log \delta+\upgamma) + \log C + \upgamma - \log {\det}_{\zeta}' \Delta_g + O(\delta) + O(e^{-\alpha C})
\]
as $\delta \rta 0$ and $C \rta \infty$, where $\alpha > 0$ depends on the manifold $(M,g)$.  
\end{theorem}

Note that, in the statements of our results and their proofs, we denote the zeta-regularized determinant of the Laplacian  by ${\det}_{\zeta} \Delta$ for manifolds with boundary and ${\det}_{\zeta}' \Delta$ for closed manifolds.  The prime notation indicates removal of the zero eigenvalue; the Laplacian has a zero eigenvalue precisely when the manifold is closed. In the discrete setting, this zero eigenvalue is directly related to the divergence of the mass of large loops (see Section~\ref{sec-discrete}); we observe this connection in the continuum setting in Remark~\ref{remark-integrand-asym}.

As in the discrete case (see Section~\ref{sec-discrete}), we can interpret Theorem~\ref{thm-loop-det} as relating the zeta-regularized determinant of the Laplacian to the partition function of the Brownian loop soup, the loop soup associated to the Brownian loop measure.  See the discussion at the end of Section~\ref{sec-relate}.  We have alluded to this loop soup interpretation in Figure~\ref{RLQGchart}. 

We stress that very close variants of Theorem~\ref{thm-loop-det}, expressed in terms of heat kernels, have appeared in different forms in several previous works---see, e.g.,~\cite[Equation (4.40)]{alvarez}.  Our main contribution is to express the expansion in the language of Brownian loops.  Our key objective in proving Theorem~\ref{thm-loop-det} is to interpret the zeta-regularized determinant of the Laplacian geometrically in terms familiar to probabilists. Such a connection between the regularized determinant of the Laplacian and the regularized mass of Brownian loops was proposed (but not rigorously proven) in \cite{sle2loopmeasure} (in the discussion after Definition 2.21) by analogy with the discrete setting, but our results are the first to rigorously establish this heuristic.  We also present a complete and self-contained proof of Theorem~\ref{thm-loop-det} that should be clear and accessible to readers with a probability background.

In Theorem~\ref{thm-loop-det}, we regularized the mass of large loops on a closed manifold by a simple cutoff $C$.  We can also consider other ways of regularizing the mass of Brownian loops and relate these different regularized loop masses to the zeta-regularized determinant of the Laplacian.  In particular, we can consider a regularization scheme that is natural in the discrete setting (as we describe in Section~\ref{sec-discrete}): namely, we can penalize loops with large quadratic variation by modifying the Brownian loop measure with a Radon-Nikodym factor that decays exponentially in the quadratic variation of the loop. In this case, we obtain

\begin{prop}\label{prop-decay}
Let $(M,g)$ be a closed manifold. Let $\mu^\tloop_{M,g,\kappa}$ denote the measure obtained from the Brownian loop measure by reweighting each loop $\eta$ by the exponential of $-\kappa/4$ times its quadratic variation.
Under this modified measure, the mass of Brownian loops on a closed manifold with quadratic variation at least $4\delta$ is equal to
\[
\frac{\Vol_g(M)}{4\pi \delta} - \frac{\chi(M)}{6} (\log \delta + \upgamma) - \log \kappa -  \log {\det}_{\zeta}' \Delta_g + \epsilon(\kappa,\delta)
\]
where $\epsilon(\kappa,\delta)$ tends to zero in the limit as first $\kappa$ and then $\delta$ is sent to zero.\footnote{One can show that this limit also holds as $\kappa,\delta \rta 0$ simultaneously, provided that $\kappa$ decays sufficiently rapidly relative to $\delta$.}
\end{prop}

\begin{remark}
\label{remark-other-loop-results}
In view of the results we describe in Section~\ref{sec-discrete} relating a discrete Laplacian to masses of discrete loops, the expansion in Theorem~\ref{thm-loop-det} resembles a result that appears in previous works for the continuum rectangular torus $\BB T$.  If the graph $\mathbb T_\eps$ is a lattice discretization of $\mathbb T$ at mesh size $\eps$, then the log determinant of the discrete Laplacian on $\mathbb T_\eps$\footnote{Here the two notions of the discrete Laplacian are essentially equivalent since the graph is regular, as we noted earlier in this section.} admits an expansion analogous to that of Theorem~\ref{thm-loop-det}, with the mesh size $\eps$ playing the role of $\delta$ \cite{chinta2010zeta, vertman2018regularized, duplantier-david-det-lap}; the constant-order term of this expansion is the logarithm of the zeta-regularized determinant of the Laplace-Beltrami operator on the continuum torus. Since we can express the log determinant of this discrete Laplacian in terms of masses of discrete loops, the expansion just described can be interpreted as an expansion of masses of discrete loops on the torus closely resembling Theorem~\ref{thm-loop-det}.

A similar expansion exists for the log determinant of the discrete Laplacian on $\ep \BB{Z}^2$ restricted to some fixed simply-connected rectilinear region~\cite{kenyon2000asymptotic}.
\end{remark}

Finally, Theorem~\ref{thm-loop-det} immediately yields an interpretation of the Loewner energy of an oriented simple smooth closed curve---introduced in \cite{wang2019energy,friz-shekhar}---in terms of the mass of Brownian loops the curve hits. The Loewner energy of a closed curve $\eta$ can be defined, e.g., as $\frac{1}{\pi} \int_{\BB C \backslash \eta} | \nabla \log |f'(z)||^2 dz^2$, where $f$ maps the two connected components of $\BB C \backslash \eta$ onto the two connected components of $\BB C \backslash \BB R$ and $f(z) - z \rta 0$ as $z \rta \infty$~\cite[Theorem 1.2]{wang2018equivalent}.

\begin{corollary}
Let $(S^2, g)$ be a sphere and $\eta$ a simple smooth closed curve on the sphere. Then the mass of loops hitting $\eta$ with quadratic variation between $4\delta$ and $4C$ is given by 
\[\frac{\Len_g(\eta)}{2\sqrt{\pi \delta}} + \log C - \log \Vol_g(S^2) -\frac{1}{12} I_L(\eta) + K +O(\delta^{1/2}) + O(e^{-\alpha C}),\]
where $I_L(\eta)$ is the Loewner energy of the curve $\eta$, $K$ is a constant independent of $g,\eta,\delta,$ and $C$, and $\alpha>0$ depends only on the metric $g$.%
\end{corollary}
\begin{proof}
The mass of loops hitting $\eta$ of quadratic variation between $4\delta$ and $4C$  is equal to the total mass of loops in this range on the sphere, minus the mass of loops in this range in each of the two connected components $D_1, D_2$ of $S^2 \backslash \eta$.  The exponential mixing of Brownian motion on the sphere implies that the total mass of loops of quadratic variation greater than $4C$ avoiding $\eta$ is $O(e^{-\alpha C})$.  Combining this observation with Theorem~\ref{thm-loop-det}, we deduce that the mass of loops hitting $\eta$ of quadratic variation between $4\delta$ and $4C$ is given by
\[\frac{\Len_g(\eta)}{2\sqrt{\pi \delta}} + \log C  - \log {\det}_{\zeta}' \Delta_{S^2,g} + \log {\det}_{\zeta} \Delta_{D_1,g} + \log {\det}_{\zeta} \Delta_{D_2,g}
+ \upgamma +O(\delta^{1/2}) + O(e^{-\alpha C}).\]
It follows from~\cite[Equation (14), Proposition 7.1, Theorem 7.3]{wang2018equivalent} that the expression $\log {\det}_{\zeta}' \Delta_{S^2,g} - \log {\det}_{\zeta} \Delta_{D_1,g} - \log {\det}_{\zeta} \Delta_{D_2,g}$ is equal to $\log \Vol_g(S^2) +\frac{1}{12} I^L(\eta)$ plus a constant.  The result follows.
\end{proof}
A similar interpretation of Loewner energy in terms of Werner's measure on self-avoiding loops \cite{werner-self-avoiding-loops} is given in \cite{wang2018note}.

\subsection*{Outline}

Section~\ref{sec-discrete} provides some context for our results from the discrete setting.  We define discrete analogues of the Laplace-Beltrami operator on graphs, and we describe connections to loop measures on graphs in the literature.  As we describe above,   Theorem~\ref{thm-loop-det} and Proposition~\ref{prop-decay} translate these results for graphs directly to the continuum setting.

Section~\ref{sec-brownian-loops} reviews the construction of the Brownian loop measure in the continuum setting; Section~\ref{sec-reg-det} defines the zeta-regularized determinant of the Laplacian.  In Section~\ref{sec-relate} we relate these two objects, proving Theorem~\ref{thm-loop-det} and Proposition~\ref{prop-decay}.

In Section~\ref{sec-square-subdivision} we prove our rigorous version of the heuristic definition of LQG (``Definition''~\ref{def-original}).  We stated this result loosely above as Theorem~\ref{thm-weighting}; we state this theorem more precisely in Section~\ref{sec-square-subdivision} (Theorem~\ref{thm-laplacian-reweighting}) after defining and motivating our quantum-regularized LQG surface model.  We end the section by interpreting our result, first in the classical regime of matter central charge $\cc < 1$, and then in the extended range $\cc \in (1,25)$.

Finally, in Section~\ref{sec-open}, we discuss open problems that arise from our work.

\noindent\textbf{Acknowledgements: } J.P.\ was partially supported
by the National Science Foundation Graduate Research Fellowship under Grant No. 1745302. M.A, M.P., and S.S.\ were partially supported by NSF Award: DMS 1712862. We thank Ewain Gwynne, Jason Miller, Peter Sarnak, Xin Sun, Yilin Wang, and Wendelin Werner for helpful conversations.

\section{The discrete setting} \label{sec-discrete}

It is possible to define several slightly different notions of a discrete Laplacian associated to a graph.  We will consider two such definitions: the \emph{graph Laplacian} and the \emph{random walk Laplacian}.

\begin{defn}[Two different notions of ``discrete Laplacian'']  
The \emph{graph Laplacian} on a graph ${\mcl G}$ with vertices $\{v_j\}_{j = 1}^{J}$ is defined as the $J \times J$ matrix $\Delta_{\op{graph}}$ with diagonal entries
\[
\Delta^{jj}_{\op{graph}}= \text{degree of the vertex $v_j$}
\]
and other entries
\[
\Delta^{jk}_{\op{graph}}=  \begin{cases}
   -1, & \text{if $v_j$ and $v_k$ are connected by an edge }\\
  0, & \text{otherwise }
  \end{cases}.
\]
The \emph{random walk Laplacian} $\Delta_{\op{RW}}$ is given by the matrix $I-P$, where $P$ is the transition probability matrix associated to simple random walk on the graph.  
\end{defn}

Note that the two notions of discrete Laplacian that we have defined are the same on a regular graph (a graph whose vertices all have the same degree) up to a constant factor that depends only on the number of vertices of the graph and their common degree.  More generally, the determinants of  $\Delta_{\op{graph}}$ and  $\Delta_{\op{RW}}$ on a graph ${\mcl G}$ with vertices $\{v_j\}_{j = 1}^{J}$ are related by
\[
\det \Delta_{\op{graph}} =  \prod_{j=1}^J d_j \det \Delta_{\op{RW}},
\]
where $d_j$ is the degree of the vertex $v_j$.  

It is expected that the partition functions of many statistical mechanics models behave asymptotically like powers of the determinant of the discrete Laplacian~\cite{bpz-conformal-symmetry}.  This highlights an important reason why LQG is so interesting to study: if LQG represents the limit of random planar maps sampled with probability proportional to powers of $\det \Delta_{\op{graph}}$ or $\det \Delta_{\op{RW}}$, then LQG should also describe the limiting behavior of random planar maps decorated by any of these statistical mechanics models.

In some cases, the connection between the statistical mechanics model and the determinant of the discrete Laplacian is not just asymptotic, but an exact combinatorial relation on the discrete level.  The heuristic correspondence between $(\det \Delta)^{-\cc/2}$ and the $\cc$-dimensional Gaussian free field in the continuum setting is based on an exact identity in the discrete setting: the partition function of the discrete Gaussian free field %
 is precisely given by $(\det \Delta_{\op{graph}})^{-1/2}$ (see, e.g.,~\cite{berestycki-lqg-notes}).  The determinant $\det \Delta_{\op{graph}}$ also encodes the number of spanning trees on the graph, by Kirchoff's matrix-tree theorem.

We devote the remainder of this section to describing another such combinatorial relation in the discrete setting: a relation between the determinant $\det \Delta_{\op{RW}}$ and masses of loops under different discrete loop measures on the graph.  The following three sections of this paper will culminate in proofs of Theorem~\ref{thm-loop-det} and Proposition~\ref{prop-decay}, which describe an analogous connection in the continuum setting between the zeta-regularized determinant of the Laplacian to the mass of loops under the Brownian loop measure.

We begin by describing two different definitions of discrete loop measures on a graph.
We define the \textit{random walk loop measure} on a graph $\mcl G$ in two stages (see, e.g.,~\cite[Chapter 9]{lawler-limic-walks}):
\begin{enumerate}
\item
We first define a measure on \emph{rooted} loops in ${\mcl G}$.  Consider the loop rooted at a vertex $x_1$ in ${\mcl G}$ that traces, in order, the vertices $x_2,\ldots,x_n$ before returning to $x_1$.  (Note that the vertices $x_j$ are not necessarily distinct, i.e., the loop does not have to be simple.)  The mass that we assign to this loop is given by the probability that a simple random walk on ${\mcl G}$ started at $x_1$ and run for $n$ steps traces exactly this loop.
\item
We then define the random walk loop measure as a measure on \emph{unrooted} loops, which are equivalence classes of rooted loops that are the same up to a shift in the choice of root (i.e., they are the same if we ``forget'' the root).  The mass that we  assign to an unrooted loop is the sum of the masses of all  distinct rooted loops in that equivalence class under the rooted loop measure, divided by the length of the loop.
\end{enumerate}
We can similarly define a continuous-time discrete loop measure on $\mcl G$~\cite{lejan-notes} in terms of a continuous-time version of simple random walk on $\mcl G$ with steps at exponential times.  

Note that the mass of loops under either of these loop measures is infinite because of the contribution of arbitrarily large loops. We could remove this divergent quantity by considering a graph $\mcl G$ with boundary and define the random walk loop measure as above in terms of random walk killed when it hits the boundary of $\mcl G$.  

In the boundary case, the quantity $- \frac{1}{2} \log \det \Delta_{\op{RW}} = - \frac{1}{2} \log \det (I-P)$ is exactly equal to the mass of loops under the random walk loop measure~\cite[Proposition 9.3.3]{lawler-limic-walks}, and it also equals the mass of loops under the continuous-time discrete loop measure that visit more than one vertex~\cite[Equation (2.5)]{lejan-notes}.  For both loop measures, this equivalence is proven by expressing the mass of loops as the trace
\[
\tr( P + P^2/2  + P^3/3 + \cdots) = -\tr( \log(I - P)),
\]
which is then shown to equal $-\log \det(I - P)$. 

We now explain that the equivalence between $\det \Delta_{\op{RW}}$ and the masses of the two discrete loop measures just described can also be phrased as identifying $(\det \Delta_{\op{RW}})^{-\cc/2}$ as the partition function of the \emph{intensity-$\cc$ loop soup} associated to these two loop measures. The intensity-$\cc$ loop soup associated to a measure on loops is an intensity-$\cc$ Poisson point process on the space of loops; i.e., a Poisson point process on the space of loops corresponding to $\cc$ times the original measure.  In general, if a measure space has total measure $\lambda$, then an intensity-$\cc$ Poisson point process on that space is a probability measure that assigns probability $e^{-\cc \lambda}$ to the empty configuration (with zero points).  Thus, if we multiply the law of the Poisson point process by the constant $e^{\cc \lambda}$, we obtain a new measure which has total mass $e^{\cc\lambda}$ and which assigns mass $1$ to the empty configuration.  The quantity $e^{\cc \lambda}$ can thus be interpreted as the {\em partition function} associated to the Poisson point process when it is normalized so that the empty configuration has weight $1$. In the setting described here, since the random loop measure has total mass $\lambda = -\frac{1}{2} \log \det \Delta_{\op{RW}}$, we have $e^{\cc \lambda} = (\det \Delta_{\op{RW}})^{-\cc/2}$. We can therefore say that an intensity-$\cc$ random walk loop soup has a partition function given by $(\det \Delta_{\op{RW}})^{-\cc/2}$.

Now, in the closed (i.e. boundaryless) case, the quantity $- \frac{1}{2} \log \det \Delta_{\op{RW}} = - \frac{1}{2} \log \det (I-P)$ is infinite because $\Delta_{\op{RW}}$ has a zero eigenvalue (corresponding to the all $1$'s vector); and, indeed, we have noted that the mass of loops under either of these loop measure is infinite in this case because of the contribution of arbitrarily large loops.  But we can still express a regularized version of the mass of loops in terms of a determinant.  For instance, we can modify the discrete loops measures  with a Radon-Nikodym factor that decays exponentially in the length of the loop---more simply, this means penalizing large loops by killing a loop with probability $1-\alpha$ for every edge it traverses, for some $\alpha \in (0,1)$.  Then the mass of loops can be computed as in the boundary case, with $P$ replaced by $\alpha P$, and the resulting regularized mass of loops is equal to $- \frac{1}{2} \log \det( I - \alpha P)$. If $\lambda_0=0, \lambda_1, \ldots, \lambda_{n-1}$ are the eigenvalues of $\Delta_{\op{RW}} = I-P$, so that $\wt \lambda_j = 1-\lambda_j$ are the eigenvalues of $P$, then $\log \det( I - \alpha P)$ is equal to $$\log \prod_{j=0}^{n-1} (1-\alpha \wt\lambda_j) = \log \prod_{j=1}^{n-1}(1-\alpha \wt\lambda_j) + \log(1-\alpha) = \log {\det}'(I-P) + \log(1-\alpha) + o(1-\alpha),$$ where $\det'(I-P)$ is the product of the non-zero eigenvalues of $I-P = \Delta_{\op{RW}}$ and $(1-\alpha)$ is the killing factor as described above. As we noted in Section~\ref{sec-intro}, we employ the same regularization scheme in the continuum setting in Proposition~\ref{prop-decay}.

\section{Defining the Brownian loop measure} \label{sec-brownian-loops}

The Brownian loop measure was defined in~\cite{lawler-werner-soup} for subsets of the complex plane equipped with the Euclidean metric, and generalized in~\cite{werner-self-avoiding-loops} to compact two-dimensional Riemann surfaces $M$ with or without boundary equipped with a metric $g$ compatible with its complex structure.  We describe the Brownian loop measure in this more general setting.  We begin by defining some measures associated to the manifold $(M,g)$; to keep the notation simple, we will denote these measures with the dependence on $(M,g)$ implicit whenever there is no possibility of confusion.

Let $\Delta$ denote the Laplace-Beltrami operator on $(M,g)$.  We can define Brownian motion on $(M,g)$ as a Markov process with transition kernel  given by the kernel $p(z,w;t)$ associated to the operator $\frac{1}{2} \Delta$; it is a probability measure on the space of continuous paths from $[0,\infty)$ to $(M,g)$, equipped with the metric of local uniform convergence.\footnote{For the purposes of convergence of measures on paths, we want to consider paths with finite domains $[0,t]$ as paths on all of $[0,\infty)$ by implicitly assuming that the interval $[t,\infty)$ is all mapped to the value of the path at $t$.} If $(M,g)$ has a boundary, then Brownian motion is killed upon hitting $\partial M$, so $\int_M p(z,w;t) \Vol_g(dw)$ may be strictly less than $1$.  %

We define the measure $\mu(z,\cdot;t)$ of Brownian motion on $(M,g)$ started at $z$ and run until time $t$ as the law on continuous paths in $(M,g)$ of $B|_{[0,t]}$, for $B$ a Brownian motion on $(M,g)$ started at $z$, restricted to the event that it lasts for at least time $t$.  This measure has total mass $\int_{y \in M} p(z,y;t)$.   In analogy with Brownian motion in $\BB{R}$,  almost every path $\eta$ in the support of $\mu(z,\cdot;t)$ has quadratic variation equal to $2t$; see, e.g.,~\cite[Corollary 5.2]{emery-bm-manifolds}.  We denote  by $\nu(\eta)$ the length of the time interval that parametrizes $\eta$.

Observe that we can decompose $\mu(z,\cdot;t)$ as
\[
\mu(z,\cdot;t) = \int_{w \in M} \mu(z,w;t) \Vol(dw)
\]
with $\mu(z,w;t)$ defined by choosing a neighborhood $N$ of $z$, restricting $\mu_{M,g}(z,\cdot;t)$ to paths $B: [0,t] \rta M$ satisfying $B_t \in N$, normalizing this restricted measure by $\Vol(N)$, and taking a weak limit of these measures as $N$ shrinks to the point $z$.  %
We then define the \emph{Brownian bridge measure} $\mu(z,w)$ as
\eqb
\mu(z,w) := \int_0^\infty \mu(z,w;t) dt
\label{eqn-mu-def}
\eqe

\begin{lem}
For fixed $z,w \in M$ (not necessarily distinct) and conformally equivalent metrics $g, \wh g$, the measures $\mu_{M,g}(z,w)$ and $\mu_{M,\wh g}(z,w)$ agree modulo time-parametrization of the curves.
\end{lem}

\begin{proof}
Consider first the simplest version of this theorem, when $M$ has boundary and $z \neq w$. The proof of the lemma in this case is almost identical to that in the case of a simply connected domain embedded in $\C$ \cite[Section 3.2.1]{lawler-werner-soup}, and follows from two equalities: Firstly, the normalized measures $\mu_{M,g}(z,w)/|\mu_{M,g}(z,w)|$ and $\mu_{M,\wh g}(z,w)/|\mu_{M,\wh g}(z,w)|$ agree modulo time-parametrization of the curves, since Brownian motion modulo time-parametrization does not depend on the conformal factor of the metric. Secondly, $|\mu_{M,g}(z,w)| = |\mu_{M,\wh g}(z,w)|$; that is, the Green functions $G_{g}(z,w), G_{\wh g} (z,w)$ agree. Roughly speaking, this is true because of Brownian scaling. In the definition of the Green function, modifying $g$ by a conformal factor affects time-parametrization in the opposite way as it does space scaling, and these factors exactly cancel out. Combining these two equalities yields the simplest case of this theorem.

Now consider the general case, where one of the following holds: 1) $M$ has no boundary, or 2) $z=w$. The above argument is not directly applicable since the measures $\mu_{M,g}(z,w)$ and $\mu_{M,\wh g}(z,w)$ are infinite. Indeed, for 1) the mass of long curves is infinite, and for 2) the mass of small loops is infinite. We will thus need to perform some kind of truncation. For 1), let $D \subset M$ be some small region not containing $z,w$, then by the above argument, the measures $\mu_{M,g}(z,w)$ and $\mu_{M,\wh g}(z,w)$ restricted to curves not hitting $D$ agree. Shrinking $D$ down to a point, we are done. For 2), let $U$ be a neighborhood of $z$, then the above argument shows that the measures $\mu_{M,g}(z,w)$ and $\mu_{M,\wh g}(z,w)$ restricted to curves not contained in $U$ agree. Shrinking $U$ down to $\{z\}$ yields the result. Finally, if both 1) and 2) are applicable, we perform both truncations above.
\end{proof}

A priori, one might construct a measure on Brownian loops on $(M,g)$ by integrating the measure $\mu(z,z)$ over $M$ with respect to the area measure $\Vol$ and then ``forgetting'' the roots of the loops, i.e., considering the induced measure on unrooted loops.   The problem with this naive approach is that, unlike $\mu(z,z)$ for fixed $z$, the integrated measure  $\int_{M} \mu(z,z) \Vol(dz)$ is \emph{not} conformally invariant.  Roughly speaking, this integrated measure weights a given loop $\eta$ by the ``size'' of the set of possible roots $z \in M$ for that loop; and that ``size'' is given by $\nu(\eta)$, which depends on the choice of metric.  Thus, to obtain a conformally invariant Brownian loop measure, one must normalize each loop $\eta$ in the support of the Brownian bridge measure by $\nu(\eta)$ before integrating over the manifold.  More precisely, we define 

\begin{definition}
The \emph{rooted Brownian loop measure} on $(M,g)$, denoted $\mu^{\text{rooted}}$, is a measure on rooted loops, i.e., paths $\eta: [0,L_\eta] \rta M$ with $\eta(0) = \eta(L_\eta)$, given by
\[
\mu^{\text{rooted}} := \int_M \frac{1}{\nu(\eta)} \mu(z,z) \Vol(dz).
\]
Here, $\frac{1}{\nu(\eta)} \mu(z,z)$ is the measure obtained by weighting each loop $\eta \sim \mu(z,z)$ by $\frac{1}{\nu(\eta)}$.
An unrooted loop is an equivalence class of rooted loops under the equivalence relation identifying $\eta$ with \[
\theta_r \eta(s) :=
\begin{dcases}
\eta(s+r)
,\quad &\text{if} \: s\leq L_\eta-r   \\
\eta(s+r-L_\eta),\quad  &\text{if} \: s>L_\eta-r
\end{dcases}
\]
for $r \in [0,L_\eta]$. The \emph{Brownian loop measure} $\mu^{\text{loop}}$ on $(M,g)$ is the measure on unrooted loops induced by $\mu^{\text{rooted}}$.
\end{definition}

The following lemma asserts that $\mu^{\text{loop}}$ is a conformally invariant measure on unrooted loops:

\begin{lem}
The Brownian loop measure is conformally invariant.
\end{lem}

\begin{proof}
Let $g,\wh g$ be conformally equivalent metrics on a manifold $M$.
As noted in~\cite[Section 4.1]{lawler-werner-soup} in the Euclidean context, if $T$ is a bounded measurable function on rooted loops satisfying $\int_0^{\nu(\eta)} T(\theta_r \eta) dr = 1$ for all rooted loops $\eta$, then $ \int_M T(\eta)^{-1} \mu(z,z) \Vol_g(dz)$ induces the measure $\mu^{\text{loop}}$ on unrooted loops.  Writing $\wh g = e^{2 \sigma} g$, the result follows from setting $T(\eta) = \frac{e^{2 \sigma(\eta(0))}}{\int_0^{\nu(\eta)} e^{2 \sigma(\eta(t))} dt}$.
\end{proof}

\section{Defining the zeta-regularized determinant of the Laplacian}\label{sec-reg-det}

It is well known that the Laplace-Beltrami operator $\Delta$ on a two-dimensional compact Riemannian manifold has a discrete spectrum
\[0 \leq \lambda_0 \leq \lambda_1 \leq \cdots. \]
The determinant of the Laplacian should correspond in some sense to the product of the nonzero eigenvalues; however, this is not rigorously defined since the spectrum is unbounded.  Alvarez~\cite{alvarez} interprets this quantity in the case of manifolds with boundary as
\eqb
\log \det \Delta = - \int_\epsilon^\infty t^{-1} \tr(e^{-t\Delta}) dt,
\label{logdet_alvarez}
\eqe
with the $\epsilon$ called an ``ultraviolet cutoff''. This quantity diverges as $\epsilon \to 0$. 

However, the standard approach to regularizing the formal sum ``$\log \det \Delta = \sum_{\lambda_j \neq 0} \log \lambda_j$'' is \emph{zeta regularization}.  To describe this approach, we introduce the following definition:

\begin{defn}
\label{def-zeta_def}
The Minakshisundaram--Pleijel zeta function is the zeta function associated to the Laplace-Beltrami operator of a compact Riemannian manifold; i.e., it is the function $\zeta$ defined for $s \in \C$ with $\Re s > 1$ as
\eqb
\zeta(s) = \sum_{\lambda_j \neq 0} \lambda_j^{-s}. \label{zeta_def}
\eqe
and as the analytic continuation of this expression for other values of $s \in \C$.
\end{defn}

The analytic continuation of this Dirichlet series for a general manifold setting with arbitrary dimension was first studied in \cite{mp-zeta-function}. Note that $\zeta$ is a well-defined analytic function on the domain  $\Re s > 1$ since (by Weyl's formula~\cite{Weyl1911}) the eigenvalues grow asymptotically at a linear rate, guaranteeing the convergence of the series~\eqref{zeta_def}.  We interpret the function $\zeta$ for $\Re s > 1$ directly in terms of Brownian loops in Proposition~\ref{prop-zeta-brownian-loops}.

Observe that, for $\Re s > 1$, the derivative of $\zeta$ is given by
\eqb
\zeta'(s) = - \sum_{\lambda_j \neq 0} \lambda_j^{-s} \log \lambda_j \label{eqn-zeta-deriv}
\eqe
If we formally substitute $s=0$ into the right-hand size of~\eqref{eqn-zeta-deriv}, we get the formal sum ``$\log \det \Delta = \sum_{\lambda_j \neq 0} \log \lambda_j$'' that we are trying to regularize.  Though the right-hand size of~\eqref{eqn-zeta-deriv} is not strictly defined for $s=0$, we \emph{have} defined the left-hand size of~\eqref{eqn-zeta-deriv} in Definition~\ref{def-zeta_def} by analytic continuation.  This motivates the following definition.

\begin{definition}
The \emph{zeta-regularized determinant of the Laplacian} (or the \emph{functional determinant}) on a compact manifold with or without boundary is $e^{-\zeta'(0)}$, where $\zeta$ is the Minakshisundaram--Pleijel zeta function defined in \eqref{zeta_def}.
\label{def-zeta-det}
\end{definition}

In preparation for the next section, in which we relate the zeta-regularization of $\det \Delta$ to Brownian loops, we now describe the standard approach for deriving an explicit expression for this derivative $\zeta'(0)$---proving in the process that $\zeta$ indeed extends analytically to a neighborhood of $s=0$.  We first express $\zeta$ for $\Re s > 1$ in terms of the \emph{Mellin transform} of $\tr(e^{-t\Delta}) $~\cite{mellin}:
\eqb \label{eq_mellin}
\zeta(s) = \sum_{\lambda_j \neq 0} \lambda_j^{-s} = \sum_{\lambda_j \neq 0} \frac{1}{\Gamma(s)} \int_0^
\infty t^{s-1} e^{-t \lambda_j} \ dt = \frac{1}{\Gamma(s)} \int_0^\infty t^{s-1} (\tr (e^{-t \Delta}) - n) dt.
\eqe
Here, $n \in \{0,1\}$ denotes the multiplicity of the zero eigenvalue; i.e., $n=1$ for closed manifolds, and $n=0$ for manifolds with boundary.  
The expression~\eqref{eq_mellin} is not defined for $s=0$ since the integrand $ t^{s-1} (\tr (e^{-t \Delta}) - n)$ blows up too fast near $t=0$.  However, we can express~\eqref{eq_mellin} as an analytic function that \emph{is} defined in a neighborhood of $s=0$.  

The key ingredient in deriving this alternative expression is the \emph{short-time expansions} of the trace $\tr (e^{-t \Delta})$ of the heat kernel for manifolds with and without boundary, originally derived by~\cite{mckean-singer} (see also \cite[Section 1]{osgood-phillips-sarnak}).  For closed manifolds, it is given by
\eqb \tr (e^{-t\Delta} ) = \frac{\Vol_g(M)}{4\pi t}  + \aconst + O(t); \label{eqn-short-time-closed} \eqe
and, for manifolds with boundary, it is given by
\eqb \tr(e^{-t\Delta}) = \frac{\Vol_g(M)}{4\pi t} - \frac{\Len_g(\partial M)}{8\sqrt{\pi t}} + \aconst + O(t^{1/2}).\label{eqn-short-time-boundary} \eqe
These short-time expansions allow us to express the quantity $\tr (e^{-t \Delta}) - n$---in both the closed manifold and manifold with boundary cases---as the sum of an expression of the form \eqb a/t + b/\sqrt{t} + c \label{eq-short-time} \eqe and a quantity that decays sufficiently fast as $t \rta 0$.  
 This decomposition allows us to write  the integral  in~\eqref{eq_mellin} as the sum of a convergent integral and an integral of the form
\[
 \frac{1}{\Gamma(s)} \int_0^\ep t^{s-1} (a/t + b/\sqrt{t} + c) dt,
\]
for some fixed $\ep > 0$,\footnote{One could just take $\ep$ to equal, say, $1$. This expression also appears in the proofs of Theorem~\ref{thm-loop-det} and Proposition~\ref{prop-decay} with $\ep = \delta/2$.}
which evaluates to
\[
 \frac{a}{(s-1) \Gamma(s)} \ep^{s-1}  +  \frac{b}{(s - 1/2) \Gamma(s)} \ep^{s-1/2}  + \frac{c}{\Gamma(s+1)} \ep^s
\]
where we have used the identity $s \Gamma(s) = \Gamma(s+1)$.  The latter expression clearly extends analytically to a neighborhood of $s=0$, using the well known fact that $\Gamma$ is a meromorphic function on $\mathbb C$ (with simple poles at non-positive integers) and $\Gamma(s)^{-1}$ is analytic in a neighborhood of zero.  Hence, we have obtained an explicit expression for $\zeta$ that extends analytically to a neighborhood of $s=0$, and we can define $\zeta'(0)$ explicitly by differentiating this expression at $s=0$.

\section{Relating the regularized Laplacian determinant and Brownian loop mass} \label{sec-relate}

We now establish connections between our regularized notions of the determinant of the Laplacian and the mass of Brownian loops. Recall that for a Brownian loop $\eta$ we write $\nu(\eta)$ to denote half the quadratic variation of $\eta$. The first connection is easy to establish:

\begin{prop}\label{prop-reg-log-det}
For manifolds with boundary, the total mass of loops $\eta$ with $\nu(\eta)>\delta$ under the Brownian loop measure can be expressed as
\[
\int_\delta^{\infty} t^{-1} \tr(e^{-t\Delta/2}) dt.
\]
This expression is the negative of~\eqref{logdet_alvarez} with $\ep = \delta/2$, i.e., the negative of the logarithm of the regularized determinant of the Laplacian in the Alvarez ``ultraviolet cutoff'' normalization with cutoff $\ep = \delta/2$.
\end{prop}

\begin{proof}
The heat kernel $p$ has the eigenfunction expansion~(see, e.g., \cite[Section 2.6]{heat-kernel})
\[
p(z,w;t) = \sum_{j=0}^{\infty} e^{- \lambda_j t/2} \varphi_j(z) \varphi_j(w)
\]
Hence
\[
\int_M t^{-1} p(z,z;t) \Vol(dz) = \sum_{j=0}^{\infty} e^{- \lambda_j t/2} =  \tr(e^{-t\Delta/2});
\]
integrating over $t > \delta$, we deduce that the $\mu^{loop}$-mass of $\{ \eta \: : \: \nu(\eta) > \delta\}$ is
\[
\int_\delta^{\infty} t^{-1} \tr(e^{-t\Delta/2}) dt = \int_{\ep}^{\infty} u^{-1} \tr(e^{-u\Delta}) du
\]
with $\ep = \delta/2$.
\end{proof}

\begin{remark}
\label{remark-integrand-asym}
From the result of Proposition~\ref{prop-reg-log-det}, we immediately see that the Brownian loop measure assigns infinite mass to arbitrarily small loops from the fact that the integrand $t^{-1} \tr(e^{-t\Delta/2})$ blows up near $0$.  Moreover, from the behavior of the integrand near $\infty$, we immediately see that the Brownian loop measure assigns infinite mass to arbitrarily \emph{large} loops precisely when $\Delta$ has a zero eigenvalue---namely, when the manifold is closed.
\end{remark}

We now relate the regularized Brownian loop mass to  the zeta-regularized determinant of the Laplacian. We showed in Proposition~\ref{prop-reg-log-det} that the regularized mass of loops can be written in terms of the trace of the operator $e^{-t\Delta/2}$, which also appears (with a time change) in the Mellin transform representation~\eqref{eq_mellin} of the $\zeta$ function.  In fact, in the case of manifolds with boundary (for which the Laplacian has no zero eigenvalue),~\eqref{eq_mellin} immediately yields a Brownian loop interpretation of the $\zeta$ function:

\begin{prop}[Brownian loop interpretation of the zeta function]
\label{prop-zeta-brownian-loops}
For manifolds with boundary and each $s$ with $\Re s > 1$, the complex number $\zeta(s)$ defined in \eqref{zeta_def} is the measure of Brownian loops, with each loop weighted by the $s$-th power of its quadratic variation divided by $4^s \Gamma(s)$.
\end{prop}

To relate $-\zeta'(0)$ to our notion of regularized loop mass, we implement the approach described at the end of the Section~\ref{sec-reg-det} to derive an explicit expression for $\zeta'(0)$, and then we relate this explicit expression to the integral form of the truncated mass of Brownian loops that we proved in Proposition~\ref{prop-reg-log-det}.  As we noted in Section~\ref{sec-reg-det}, we will use the short-time expansions~\eqref{eqn-short-time-closed} and~\eqref{eqn-short-time-boundary} of the trace of the heat kernel for manifolds with and without boundary.  In both of the following proofs, we use the fact we mentioned in Section~\ref{sec-brownian-loops} that the quadratic variation of $\mu^{\text{loop}}$-a.e. loop is double the length of its parametrizing interval.

\begin{proof}[Proof of Theorem~\ref{thm-loop-det}, boundary case]
Combining the short-time expansion~\eqref{eqn-short-time-boundary} with \eqref{eq_mellin} with $n=0$ yields, for $\Re s>1$,
\alb
\zeta(s) =& \frac{1}{\Gamma(s)} \int_{\delta}^\infty t^{s-1} \tr(e^{-t \Delta}) dt + \frac{1}{\Gamma(s)} \int_0^{\delta} t^{s-1} \left( \tr(e^{-t\Delta}) - \frac{\Vol_g(M)}{4\pi t} + \frac{\Len_g(\partial M)}{8\sqrt{\pi t}} - \aconst  \right) dt \\
&+ \frac{1}{\Gamma(s)} \left(\frac{\Vol_g(M)}{4\pi (s-1)} \delta^{s-1} - \frac{\Len_g(\partial M)}{8\sqrt{\pi} (s-\frac12)} \delta^{s-\frac12}  \right) +\frac{1}{\Gamma(s+1)} \delta^s\aconst.
\ale
Here we have used the identity $s\Gamma(s) = \Gamma(s+1)$. The latter expression for $\zeta$ is analytic in a region that includes both the half-plane  $\Re s>1$ and a neighborhood of the origin, so $\zeta'(0)$ is given by the derivative of the above at $s=0$.  Since $\lim_{s\to 0} s\Gamma(s) = 1$ and $\frac{d}{ds}\big|_{s=0} \frac{1}{\Gamma(s+1)} = \upgamma$, we get
\alb
\zeta'(0) =& \int_{\delta}^\infty t^{-1} \tr(e^{-t\Delta}) \ dt + \int_0^{\delta} O(t^{-1/2}) dt - \delta^{-1} \frac{\Vol_g(M)}{4\pi} +2 \delta^{-1/2} \frac{\Len_g(\partial M)}{8\sqrt{\pi}} \\
&+ \log \delta  \aconst + \upgamma \aconst.
\ale
By Definition~\ref{def-zeta-det}, $\zeta'(0) = -\log \det_{\zeta} \Delta$.  And the first term on the right-hand side is a regularized Brownian loop mass:
\[
\int_{\delta}^\infty t^{-1} \tr (e^{-t\Delta}) dt = \int_{2\delta}^\infty u^{-1} \tr (e^{-u \Delta/2}) du,
\]
which equals the mass of Brownian loops with quadratic variation greater than $4\delta$.  This completes the proof.

\end{proof}
We now prove the theorem for closed manifolds, for which we also truncate loops $\eta$ with $\nu(\eta) > C$. The proof is similar to the above.
\begin{proof}[Proof of Theorem~\ref{thm-loop-det}, closed case]
Combining the short-time expansion~\eqref{eqn-short-time-boundary} with \eqref{eq_mellin} with $n=1$ yields, for $\Re s>1$,
\alb
\zeta(s) =& \frac{1}{\Gamma(s)} \int_{\delta}^\infty t^{s-1} (\tr(e^{-t \Delta}) -1) dt + \frac{1}{\Gamma(s)} \int_0^{\delta} t^{s-1} \left( \tr(e^{-t\Delta}) - \frac{\Vol_g(M)}{4\pi t}  - \aconst  \right) dt \\
&+  \frac{\Vol_g(M)}{4\pi (s-1)\Gamma(s)} \delta^{s-1} +\frac{1}{\Gamma(s+1)} \left( \aconst - 1\right)\delta^s.
\ale
Written in the above form, it is clear that $\zeta(s)$ extends holomorphically to a neighborhood of $s=0$. Since $\lim_{s\to 0} s\Gamma(s) = 1$ and $\frac{d}{ds}\big|_{s=0} \frac{1}{\Gamma(s+1)} = \upgamma$, we have
\eqb \label{eqn_noboundary}
\zeta'(0) = \int_{\delta}^\infty t^{-1} (\tr(e^{-t\Delta})-1) \ dt - \delta^{-1} \frac{\Vol_g(M)}{4\pi} + (\log \delta +\upgamma) \left(\aconst-1\right)+O(\delta).
\eqe
The term $\int_{\delta}^\infty t^{-1} ( \tr (e^{-t\Delta}) - 1) dt$ is not quite a regularized Brownian loop mass because of the $-1$ term; we cannot split the integral directly to remove this term since both the integrals $\int_{\delta}^\infty t^{-1} dt$ and $\int_{\delta}^\infty t^{-1}  \tr (e^{-t\Delta}) dt$ diverge.  The latter integral diverges because the Laplacian has a zero eigenvalue---or, from another point of view, because of the infinite mass of arbitrarily large Brownian loops.  (See Remark~\ref{remark-integrand-asym}.)
 We therefore truncate the upper limit of integration:
\[
\int_{\delta}^\infty t^{-1} ( \tr (e^{-t\Delta}) - 1) \,dt = \int_{2\delta}^{2C} u^{-1} (\tr (e^{-u\Delta/2}) - 1) \,du + \int_{2C}^\infty u^{-1} (\tr (e^{-u\Delta/2}) - 1) \,du \]
This equals the $\mu^\tloop_{M,g}$-mass of loops with quadratic variation between $4\delta$ and $4C$, plus
\[
\log \delta - \log C + O(e^{-\alpha C}).
\]
(Note that we have used the fact that, by exponential mixing of Brownian motion, there exists an $\alpha > 0$ for which $\tr e^{-u\Delta/2} = 1 + O(e^{-\alpha u})$ as $u \rta \infty$.)  Substituting the latter expression into~\eqref{eqn_noboundary} and recalling that $\log \det_{\zeta} \Delta = -\zeta'(0)$ completes the proof.
\end{proof}

Finally, we prove Proposition~\ref{prop-decay}, which relates a different regularization of the mass of Brownian loops on a closed manifold to the zeta-regularized determinant of its Laplacian.

\begin{proof}[Proof of Proposition~\ref{prop-decay}]
Our starting point for the proof is equation~\eqref{eqn_noboundary}.  We can express the integral $\int_{\delta}^\infty t^{-1} (\tr(e^{-t\Delta})-1) \ dt$ on the right-hand side of~\eqref{eqn_noboundary} as the sum
\eqb \label{eq-decay-loops}
\int_{\delta}^\infty t^{-1} e^{-\kappa t} \tr (e^{-t\Delta}) \ dt - \int_{\delta}^\infty t^{-1} e^{-\kappa t}\ dt +  \int_{\delta}^\infty t^{-1} (1-e^{-\kappa t})( \tr (e^{-t\Delta})-1) \ dt.
\eqe
We examine each of the terms of \eqref{eq-decay-loops} in turn.  The first term is equal to
\[
\int_{2\delta}^\infty u^{-1} e^{-\kappa u/2} \tr (e^{-u\Delta/2}) \ du,
\]
which, in turn, is equal to the $\mu^\tloop_{M,g,\kappa}$-mass of loops with quadratic variation larger than $4\delta$.  (Here we are using the fact from Section~\ref{sec-brownian-loops} that  the quadratic variation of a loop sampled from $\mu^\tloop_{M,g,\kappa}$ is almost surely equal to double the length of its parametrizing interval.)

Next, using an expansion of the incomplete Gamma function, we have
\[
- \int_{\delta}^\infty t^{-1} e^{-\kappa t}\ dt  = \upgamma + \log (\delta \kappa) + O(\delta \kappa).
\]
Finally, we split the third term of~\eqref{eq-decay-loops} into two parts and bound them separately. For $t \in (\delta, -\log \kappa)$ we have $1 - e^{\kappa t} = O(\kappa t)$, and by~\eqref{eqn-short-time-closed} we have $\tr e^{-t \Delta} - 1 = O(1 + \frac1t)$.  Therefore,
\[
\int_{\delta}^{-\log \kappa} t^{-1} (1-e^{-\kappa t})( \tr (e^{-t\Delta})-1) \ dt = O(\delta^{-2}\kappa (\log \kappa)^2).
\]
And, by the exponential mixing of Brownian motion, we can choose $\alpha > 0$ such that $\tr (e^{-t \Delta}) - 1 = O(e^{-\alpha t})$ for all large $t$. Thus,
\[
\int_{-\log \kappa}^{\infty} t^{-1} (1-e^{-\kappa t})( \tr (e^{-t\Delta})-1) \ dt = O(e^{\alpha \log \kappa}).
\]
Plugging the above four equations into~\eqref{eq-decay-loops}, we conclude that the integral $\int_{\delta}^\infty t^{-1} (\tr(e^{-t\Delta})-1) \ dt$ equals the $\mu^\tloop_{M,g,\kappa}$-mass of loops with quadratic variation larger than $4\delta$, plus
\[  \upgamma + \log \delta + \log \kappa + o(1),\]
where, for fixed $\delta$, $o(1)$ is a term that goes to zero as $\kappa \to 0$. Substituting this expression for $\int_{\delta}^\infty t^{-1} (\tr(e^{-t\Delta})-1) \ dt$ in equation~\eqref{eqn_noboundary} completes the proof.
\end{proof}

Finally, we discuss the heuristic interpretation of $(\det_\zeta \Delta_g)^{-\cc/2}$ as the partition function of the intensity-$\cc$ Brownian loop soup on a closed manifold $(M, g)$ (the case with boundary is similar). Recall from Section~\ref{sec-intro} that the intensity-$\cc$ Brownian loop soup is an intensity-$\cc$ Poisson point process sampled from the loop measure $\mu^\tloop$ on $(M,g)$.

 Suppose that we have a probability measure on closed manifolds $(M,g)$ with $\Vol_g(M) = A$ a.s. for some constant $A$, and we want to weight each $(M,g)$ by the partition function of the intensity-$\cc$ Brownian loop soup on $(M,g)$. As in the discrete setting (see the end of Section~\ref{sec-discrete}), we want to interpret the intensity-$\cc$ Brownian loop soup partition function on $(M,g)$ as $\exp(\cc |\mu^\tloop_{M,g}|)$, but now $\mu^\tloop_{M,g}$ is an infinite measure so $\exp(\cc |\mu^\tloop_{M,g}|) = \infty$. We consider instead a ``truncated partition function'', in which we replace $|\mu^\tloop_{M,g}|$ with the $\mu^\tloop_{M,g}$-mass of loops with quadratic variation between $4\delta$ and $4C$. Theorem~\ref{thm-loop-det} gives an expansion of this loop mass whose divergent terms depend only on $A, \delta$ and $C$ (and not on the particular choice of $(M,g)$), and can therefore be absorbed by the normalizing constant. Thus, if we weight our probability measure on manifolds by the truncated partition function, each surface $(M,g)$ is weighted by $Z^{-1}(1+o(1))(\log \det'_\zeta \Delta_g)^{-\cc/2}$. If we had sufficient uniform control on the $o(1)$ error over the set of manifolds $(M,g)$ in the support of our probability measure, then in the $\delta \to 0, \,\, C \to \infty$ limit we would indeed be weighting each surface by $Z^{-1} (\log \det'_\zeta \Delta_g)^{-\cc/2}$. (This would certainly be the case if the probability measure on manifolds were supported on a {\em finite} collection of manifolds, e.g., manifolds obtained as smoothed versions of random planar maps.) See Question~\ref{ques-loop-version}.
 
In the more general setting in which we do not fix the volume of the surfaces we consider, we can instead perform a weighting by the \emph{volume-adjusted} mass of Brownian loops. That is, we weight each manifold $(M,g)$ by $\exp(\cc (\text{truncated loop mass}))$ \emph{divided} by a term that depends on $\Vol_g(M)$,  $\delta$ and $C$.  The latter term corresponds to the the higher-order terms in the expansion of Theorem~\ref{thm-loop-det}.

We can adapt the above to the case of manifolds with boundary, in which we must also account for the boundary term.

\section{Regularizing LQG surfaces via square subdivisions} \label{sec-square-subdivision}

In this section we prove a rigorous version of the heuristic definition of LQG we described in Section~\ref{sec-intro} (``Definition''~\ref{def-original}) in terms of a regularized version of the heuristic LQG metric.  
As we described in Section~\ref{sec-intro}, the key to our approach in regularizing the heuristic LQG metric is regularizing the underlying field $h$ at a \emph{quantum (LQG) scale} rather than a Euclidean scale; i.e., by mollifying $h$ at each point at a scale that is roughly uniform in size, not with respect to the Euclidean metric, but with respect to the ``LQG metric''.  

We have divided this section into four parts:
\begin{itemize}
\item
We begin in Section~\ref{sec-reg-intuition} by motivating our choice to regularize the field on a quantum scale.
\item In Section~\ref{section-approx-metric}, we define our quantum-scale regularized version of the heuristic LQG metric.  
\item In Section~\ref{section-approx-surface}, we use this quantum-regularized metric to define a closed manifold approximation of LQG, and we prove Theorem~\ref{thm-laplacian-reweighting} (which we loosely stated in Section~\ref{sec-intro} as Theorem~\ref{thm-weighting}) that links this particular approximation to the original heuristic definition of LQG.
\item In Section~\ref{sec-interpretation1}, we interpret our result for $\cc \in (-\infty,1)$ by relating our approximation to a (deterministic volume) quantum sphere.
\item In Section~\ref{sec-interpretation2}, we describe the implications of our theorem in the case of matter central charge $\cc \in (1,25)$.
\end{itemize}

\subsection{Motivation for regularizing on a ``quantum scale''}
\label{sec-reg-intuition}

Before describing our regularization method explicitly, we will elaborate on the motivation for mollifying the field on a quantum rather than a Euclidean scale.

The most obvious and common approach to regularizing LQG is mollifying the field $h$ on a Euclidean scale---replacing $h$ by a function $h_\ep$ that may represent the average of $h$ on a Euclidean circle of radius $\ep$~\cite{shef-kpz}, or the convolution of $h$ with $\theta(\cdot/\ep)$ for an appropriate fixed non-negative measure $\theta$ (see~\cite{shamov-gmc,berestycki-gmt-elementary}). This is how the LQG measure is generally defined: for such choices of $h_\ep$, the LQG measure associated to $h$~\cite{kahane} is defined\footnote{The Gaussian multiplicative chaos measure associated to the field $h$ is defined in the same way but with the factor $\ep^{\gamma^2/2}$ replaced by $\exp(-\frac{\gamma^2}2 \E[h_\ep(z)^2])$ in \eqref{def-measure} \cite{rhodes-vargas-review}. This amounts to multiplying the planar measure $\mu_h$ by a deterministic Radon-Nikodym derivative.} as the vague limit in probability of the regularized measures
\eqb
\mu_h^{\epsilon}(z) := \epsilon^{\gamma^2/2} e^{ \gamma h_{\epsilon}(z)} \Vol_{g_0}(dz),
\label{def-measure}
\eqe
where the {\em coupling constant} $\gamma$ is defined as the unique solution in $(0,2]$ of the equations\footnote{In the physics literature, the parameter $b = \frac{\gamma}{2}$ often appears instead of $\gamma$.}
 \eqb
 \cc = 25 - 6Q^2, \qquad Q = \frac{\gamma}{2} + \frac{2}{\gamma}
 \label{eqn-def-gamma-Q}
 \eqe
 The parameter $Q$ will feature prominently in this section, and is called the {\em background charge}.

As another example of regularizing on a Euclidean scale, the recent construction of the LQG distance function in the series of papers~\cite{dddf-lfpp,local-metrics,lqg-metric-estimates,gm-confluence,gm-uniqueness,gm-coord-change}  characterizes LQG as the limit of Euclidean-scale approximations of the distance function, known as $\ep$-\textit{Liouville first-passage percolation} (LFPP)~\cite{ding-goswami-watabiki, dg-lqg-dim, dddf-lfpp}. In the definition of $\ep$-LFPP, they consider not $e^{\gamma h_\ep}$ but $e^{\xi h_\ep}$, where  $\xi = \gamma/d_\gamma$ and $d_\gamma$ denotes the Hausdorff dimension of $\gamma$-LQG~\cite{gp-kpz}.

There are distinct advantages to working with Euclidean-scale approximations of the field.  First of all, they are often more tractable and technically easier to handle. This is a key reason why the papers constructing LQG distance function considered LFPP rather than another natural scheme for approximating LQG distances known as Liouville graph distance, which involves approximating the field on quantum scales; see~\cite[Remark 1.4]{lqg-metric-estimates}.  On a related note, Euclidean-scale approximations of a field behave nicely when we add a function to the field.  For example, if we approximate $h+b$ at a Euclidean scale by mollifying the field with a bump function, then we obtain the same random function as first mollifying $h$ and then adding $b$. As a result, the Euclidean scale approximations %
of the LQG measure and distance function reflect how the LQG measure and distance function transform when we add a constant to the field: when we add a constant $b$ to the field, the volume measure scales by a factor of $e^{\gamma b}$ and the distance function by a factor of $e^{\xi b}$.  (The latter is a special case of the so-called \textit{Weyl scaling} property of the LQG distance function; see, e.g.,~\cite{gm-uniqueness} for the statement of this property.) 

However, to state and prove a precise version of Theorem~\ref{thm-weighting}, it is crucial that we mollify the field $h$ instead on a quantum scale---which we will denote by $h^{\ep}$ with a superscript $\ep$ to distinguish from the Euclidean-scale case. Roughly speaking, we will regularize $h$ by taking its conditional expectation gives its averages on squares in our domain of approximately equal ``quantum size''. Superficially, this is what gives us the ``right'' exponent for the proof to work on a technical level; but there are more conceptual reasons for this approach as well.  First, we want to define a mollified version of the heuristic LQG metric, not simply the LQG measure or distance function associated to that metric.  The Euclidean scale approximations of the LQG measure and distance function cannot not arise as the measure and distance function associated to a single well-defined metric tensor: for this to be the case, the exponents  $\gamma$ and $\xi$ would have to differ by a factor of $2$.  When regularizing on a quantum-scale, this condition does hold: as we will see in Section~\ref{section-approx-metric}, the measure and distance function we get in this case are of the form $e^{2 h^\ep/Q}$ and $e^{h^\ep/Q}$, respectively. Therefore, we can view these versions of the regularized measure and distance function as corresponding to a single regularized metric of the form $e^{2 h^\ep/Q} (dx^2 +dy^2)$; see Definition~\ref{def-metric-regularized} below.

There are other ways in which the quantum-scale approach appears to be more geometrically natural.  The LQG measure and distance function are well-defined on equivalence classes of domain-field pairs known as {\em quantum surfaces}~\cite{shef-kpz,shef-zipper,wedges}:
\begin{defn}
A quantum surface is an equivalence classes of pairs $(D,h)$, where $D$ is a domain in $\BB{C}$ and $h$ a generalized function on $D$, with two such pairs $(\wt D,\wt h)$ and $(D,h)$ taken to be equivalent if $\wt h$ is equal to $h \circ f + Q \log|f'|$ for some conformal mapping $f: \wt D \rta D$. We similarly define a quantum surface with $k$ marked points as an equivalence class of tuples $(D,h,x_1,\ldots,x_k)$; the equivalence relation is defined the same way, only with the extra condition that the corresponding conformal map sends the marked points of one tuple to the marked points of the equivalent tuple.
\end{defn}
Both the LQG measure and LQG distance function are well-defined on these equivalence classes: if $(\wt D,\wt h)$ and $(D,h)$ are equivalent, then the LQG measure and distance function associated to $\wt h$ on $\wt D$ are equal to the pullback of the LQG measure and distance function associated to $h$ on $D$. 
A key feature of a quantum-regularized metric, such as that defined in Definition~\ref{def-metric-regularized}, is that it is \emph{also} well-defined on these equivalence classes: the pullback of  $e^{2 h^\ep/Q} (dx^2 +dy^2)$ under $f$ is $e^{2 (h^\ep \circ f)/Q} |f'|^2 (dx^2 +dy^2)$, which is equal to $e^{2 \wt h^{\ep}/Q} (dx^2 +dy^2)$.  In other words, the quantum-regularized LQG measure and distance function associated to $\wt h$ on $\wt D$ are equal to the pullback of the quantum-regularized LQG measure and distance function associated to $h$ on $D$.  

Another reason that the quantum-scale approach is more conceptually appealing from a geometric perspective is that thinking of the LQG metric as $e^{2h/Q}(dx^2+dy^2)$---or, equivalently, the metric tensor whose expression  in Euclidean coordinates is $\left( \begin{array}{cc} e^{h/Q} & 0\\ 0 & e^{h/Q} \end{array} \right)$---complements the heuristic definition of \textit{imaginary geometry flow lines} (as defined in~\cite{ig1}) as flow lines of the vector field $e^{i h/\chi}$, where $\chi = 2/\gamma - \gamma/2$.  These flow lines are well-defined on equivalence classes of so-called \textit{imaginary surfaces}.  This reinforces the notion---apparent from the machinery of the ``quantum zipper'' \cite{shef-zipper}---that the theories of quantum surfaces and imaginary surfaces are dual in some sense.

Finally, representing the heuristic LQG metric tensor as ``$e^{2h/Q}(dx^2+dy^2)$'' has the advantage of not explicitly invoking the parameter $\gamma$, which means that it can yield a candidate picture of LQG for values of matter central charge $\cc \in (1,25)$, for which the corresponding $\gamma$ values are complex, but $Q$ is still real and nonzero.  This observation is the basis of the description of LQG in this phase proposed in the recent paper~\cite{ghpr-central-charge}.  To be clear, their paper does not explicitly make this observation or refer to the form of the LQG metric with a factor of $2/Q$.  However, their random planar map model is based on essentially the same quantum-scale approximation of the field that we consider in this paper. We will say more about the work in~\cite{ghpr-central-charge} in Section~\ref{sec-interpretation2}.

\subsection{Defining the LQG metric approximation on the unit square}\label{section-approx-metric}

In order to precisely describe our regularization approach, we recall the definition of the GFF on the sphere; see \cite{dkrv-lqg-sphere} for further details. We can either view it as a distribution modulo additive constant or define it as a distribution by fixing some normalization.  We describe the latter approach in two stages:
\begin{description}
\item[1) Identifying the equivalence class.]
Let $\wh \C = \C \cup \{ \infty\}$ be the Riemann sphere, and let
$g_0 = e^{2 \sigma_0} dz$ be a smooth metric on $\wh \C$.  With $M:=(\wh \C, g_0)$, let $H_0^1(M)$ denote the Sobolev space  defined as the completion of the set of mean-zero smooth functions $\{f \in  C^\infty(M) \: : \: \int_\C f \ d\Vol_{g_0} = 0 \}$, with respect to the Dirichlet inner product
\[  (f_1,f_2)_{\nabla} := (2\pi)^{-1} \int_{\C} \nabla_{g_0} f_1(z) \cdot \nabla_{g_0} f_2(z) \ d\Vol_{g_0}.\]
Let
\[
\wt h := \sum_j \alpha_j f_j,
\]
for $f_j$ an orthonormal basis of $H_0^1(M)$ and $\alpha_j$ an independent collection of standard Gaussians.  This sum converges almost surely in the space of distributions. It is an immediate consequence that, for any $f \in H^1_0(M)$, the random variable $(\wt h, f)_\nabla$ is a centered Gaussian of variance $(f,f)_\nabla$. For any mean-zero function $\phi \in L^2(M)$ one can make sense of $\Delta^{-1}_{g_0} \phi \in H^1_0(M)$ \cite[Theorem 4.7]{aubin_1982}, so we may define the $L^2$ inner product
\[(\wt h, \phi) := -(\wt h, \Delta^{-1}_{g_0} \phi)_{\nabla}.\]

\item[2) Fixing the normalization.]
The distribution $\wt h$ can be understood as representing a GFF modulo additive constant; to obtain a bona fide GFF, we have to fix this additive constant. Let $\phi_0 \in C^\infty (M)$ be a smooth function with $\int_\C \phi_0 \ d \Vol_{g_0} \neq 0$. We define the GFF $h$ on $M$ normalized so that $(h, \phi_0) = 0$ by
\[(h, \phi) := (\wt h, \phi - c\phi_0) \]
with the constant $c = c(\phi)$ chosen so $\phi - c \phi_0$ has mean zero.  
\end{description}

Now, our approach to defining a regularized version of the  heuristic LQG metric on a quantum scale is, roughly speaking, to approximate the LQG metric on a domain by subdividing the domain into dyadic squares that have approximately the same ``LQG size''.  
(A \emph{dyadic square} is any square $S \subset \BB{S}$ with side length $2^{-n}$ and vertices in $2^{-n} \Z^2$ for some $n \in \Z_{\geq 0}$.)
For simplicity, we take our domain to be the flat unit square $\BB{S}:=[0,1]^2$ since it is easy to define the dyadic square subdivision for this domain.

Let $M = (\wh \C, g_0)$ be chosen so that $\sigma_0 \equiv 0$ in a neighborhood of $\BB{S}$, and let $h$ be a GFF on $M$ normalized so that $(h,K_0) = 0$, where $K_0$ is the Gaussian curvature associated to $M$.

For a square $S \subset \BB{S}$, we write $h_S := (h, |S|^{-2}\mathbbm1_S)$ for the average of $h$ on $S$. 
We can approximate the LQG area of the square $S$ in terms of this average $h_S$ by taking $h_\ep$ in the regularized LQG measure~\eqref{def-measure} to be the average of $h$ on an $\ep$-length square centered at $z$.  
 Since this mollification should satisfy continuity estimates analogous to those proved in~\cite{ghm-kpz} for the circle average mollification, we deduce heuristically that 
the area of the square $S$ should approximately equal
\eqb
|S|^{\gamma^2/2 + 2} e^{\gamma h_S} = (A_h(S))^{\gamma Q}, \label{eqn-area-approx}
\eqe
where
\eqb \label{eqn-mass-def}
A_h(S) := e^{h_{S}/Q} |S|.
\eqe
(Alternatively, one could interpret~\eqref{eqn-area-approx} as the conditional expectation of $\mu_h(S)$ given $h_S$, up to a small error; see \cite[Proposition 1.2]{shef-kpz}.) Note that, though the LQG measure is defined only for $\cc < 1$, the quantity $A_h(S)$ is well-defined whenever $Q$ is real and nonzero, i.e., whenever $\cc < 25$.  It is this quantity $A_h(S)$ that we use to define the ``LQG size'' of a square in our subdivision.

\begin{definition}[Definition of the $\ep$-square subdivision]
\label{def-squares}
For $\epsilon > 0$, we define the $\epsilon$-square subdivision $\cS^{\ep}$ of $\BB{S}$ as the set of dyadic squares $S$ in $\BB S$ with $A_h(S) \leq \ep$ and $A_h(S') > \ep$ for every dyadic ancestor $S' \supset S$.
\end{definition}

We can describe the $\ep$-square subdivision as the output of the following algorithmic procedure: starting with the singleton collection of squares $\{ \BB S \}$, we repeatedly take the largest square $S$ in the collection with $A_h(S) > \ep$ and replace it with its four dyadic children. For $\cc \leq 1$, this procedure terminates with probability one, yielding the $\ep$-square subdivision. For $\mathbf{c} \in (1,25)$, this algorithm does not necessarily terminate  (see~\cite{ghpr-central-charge} for an explanation of this phenomenon), but we can still describe the $\ep$-square subdivision as the set of squares $S$ with $A_h(S) < \ep$ that we observe when we run the algorithm for infinitely many steps. See Figure~\ref{fig-square-subdivision-realization} for simulations of $\cS^\ep$ for various values of $\cc$.

From the definition of $A_h(S)$, we immediately observe the following lemma:

\begin{lem}\label{lem-measurable}
The set of squares in the $\ep$-square subdivision $\cS^{\ep}$ is measurable w.r.t.\ the countable collection of quantities $h_S/Q$, as $S$ ranges over the set of all dyadic squares. The function that inputs this collection of quantities and outputs the $\ep$-square subdivision is independent of $Q$. 

Furthermore, if $\cS$ is a fixed finite partition of $\BB{S}$ into dyadic squares, the event that $\cS^{\ep} = \cS$ is measurable w.r.t.\ the finite collection of quantities $\{h_S/Q\}_{S \in \cS}$.  The function that inputs this collection of quantities and outputs the indicator of the event $\cS^{\ep} = \cS$ is independent of $Q$ (or equivalently $\cc$).
\end{lem}

Note that the set of squares in the $\epsilon$-square subdivision is not required to be finite; the number of squares is finite almost surely for $\cc \leq 1$, but infinite for small enough $\epsilon$ in the $\cc \in (1,25)$ regime.  %

\begin{figure}[ht!]
    \centering
    \begin{subfigure}[t]{0.3\textwidth}
        \includegraphics[width=\textwidth]{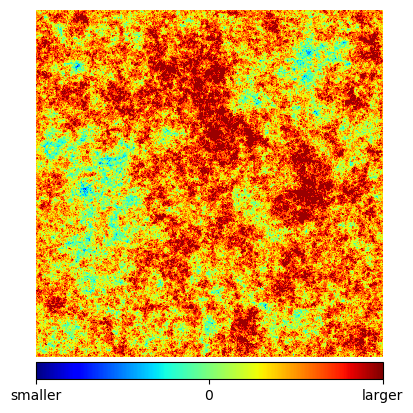}
        \caption{An instance of GFF}
    \end{subfigure}
    \quad
    \begin{subfigure}[t]{0.3\textwidth}
        \includegraphics[width=\textwidth]{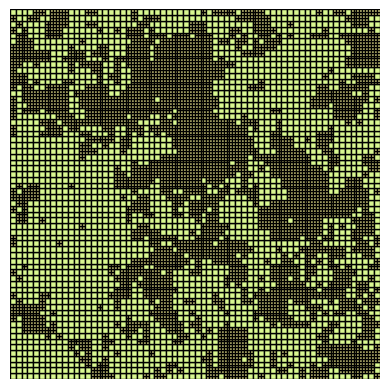}
        \caption{$\cc\approx-83$ ($\gamma=0.5$, $Q=4.25$)}
    \end{subfigure}
    \quad
    \begin{subfigure}[t]{0.3\textwidth}
        \includegraphics[width=\textwidth]{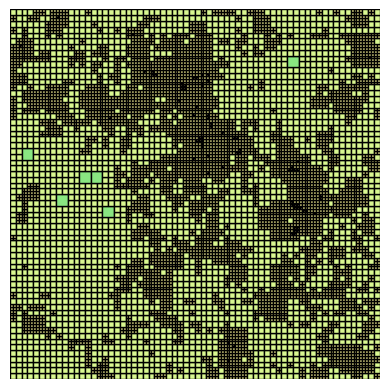}
        \caption{$\cc=-12.5$ ($\gamma=1$, $Q\approx 2.5$)}
    \end{subfigure}\\
    \begin{subfigure}[t]{0.3\textwidth}
        \includegraphics[width=\textwidth]{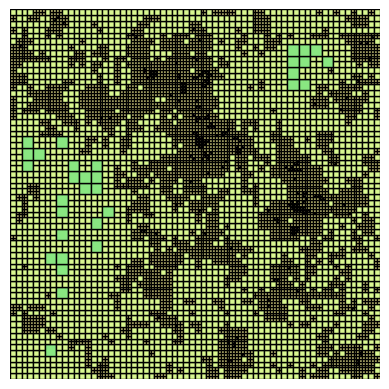}
        \caption{$\cc=0$ ($\gamma=\sqrt{8/3}$, $Q\approx 2.04$)}
    \end{subfigure}
    \quad
    \begin{subfigure}[t]{0.3\textwidth}
        \includegraphics[width=\textwidth]{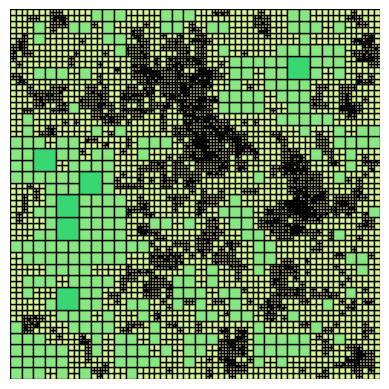}
        \caption{$\cc=19$ ($\gamma\in\C$, $Q=1$)}
    \end{subfigure}
    \quad
    \begin{subfigure}[t]{0.3\textwidth}
        \includegraphics[width=\textwidth]{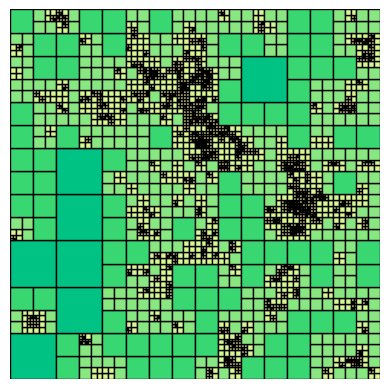}
        \caption{$\cc\approx 23.5$ ($\gamma\in\C$, $Q=0.5$)}
    \end{subfigure}
    \caption{
Illustrations of the $\epsilon$-square subdivision $\cS^{\ep}$ for different values of $\cc$ and $\ep$ equal to $2^{-12}$ times the approximate quantum area $A_h(\BB{S})$ of the unit square.  To generate the simulations in (b)-(f), we approximated the unit square by a $2^{13}\times 2^{13}$ grid, and we took the field $h$ to be a discrete Gaussian free field on this grid with Dirichlet boundary conditions.  We have drawn an instance of this field in (a).  In (b)-(f), we we have drawn the square subdivisions generated by this instance of the field, with the squares colored according to their Euclidean size.  The subdivisions (e) and (f) correspond to values of $\cc$ greater than $1$; for such $\cc$, we obtain infinitely many squares with positive probability that tends to $1$ as $\ep\rightarrow 0$.  We discuss this regime further in Section~\ref{sec-interpretation2} in the context of the relevant work of~\cite{ghpr-central-charge}.
} \label{fig-square-subdivision-realization}
\end{figure}

Roughly speaking, we define our quantum-regularized metric  by replacing $h$ with its conditional expectation  given the averages $h_S$ of $h$ on the set of squares $S$ in $\cS^{\ep}$.  To make this definition more explicit, we introduce the following notation for a region $U \subset \BB{S}$: we let $\tilde{\mathbbm{1}}_S$ denote the mean-zero measure obtained by subtracting from $\mathbbm{1}_S$ a multiple of the measure $K_0$. 

We now define the regularized version of $h$ that we use to define our quantum-regularized metric:

\begin{definition}
\label{def-hS}
Let $\cS$ be a partition of $\BB S$ into dyadic squares.  We define $h^{\cS}$ as the conditional expectation of $h$ given the averages of $h$ on the squares in $\cS$. More explicitly, we define $h^{\cS}$ as the sum of two terms:
\begin{enumerate}
\item the minimal Dirichlet energy function in $H_0^1(\BB C)$ whose $L^2$-inner product with $\tilde{\mathbbm{1}}_S$ equals $(h,\tilde{\mathbbm{1}}_S)$ for each $S \in \cS$, and
\item
a constant chosen so that $h^\cS$ has the same normalization as $h$---i.e., $(h^{\cS},K_0) = 0$.
\end{enumerate}
\end{definition}

By definition of the GFF, we can write the first term in Definition~\ref{def-hS} as a finite linear combination of functions in $H_0^1(\BB C)$ whose coefficients are independent standard Gaussian random variables.

\begin{lem}
\label{lem-alpha-law}
We can express $h^{\cS}$ as a constant plus
\eqb
\sum_{S \in \cS} \alpha^{\cS}_S f^{\cS}_S
\label{eqn-def-hS}
\eqe
where $f^{\cS}_S$ is a $(\cdot,\cdot)_\nabla$- orthonormalization of the set of functions $\{ \Delta^{-1} \tilde{\mathbbm{1}}_S \}_{S \in \cS}$, and the $\alpha^{\cS}_S$ are jointly independent standard Gaussians.  In particular, the function $h^{\cS}/Q$  is a $Q$-independent measurable function of the variables $\alpha^{\cS}_S/Q$.
\end{lem}

We will define a regularized LQG metric in terms of the function $h^{\cS^{\ep}}$, which we denote simply by $h^\ep$ to simplify notation.  To motivate this definition, we observe that, for $\cc<1$ and fixed $\ep>0$, we can approximate the LQG measure restricted to a square $S$ in  the $\ep$-square subdivision by
\[
|S|^{\gamma^2/2} e^{\gamma h_S} = ( A_h(S))^{\gamma^2/2} e^{2 h_S/Q} = ( A_h(S))^{\gamma^2/2} e^{2 h^\ep_S /Q},
    \]
Since $A_h(S)$ is approximately equal to $\ep$ (in a sense we will not make precise), the LQG  measure is approximately equal to an $\ep$-dependent constant times $e^{2 h^{\ep}_S/Q}$. In other words, the LQG  measure is approximated  by the measure induced by the metric
\eqb\label{eq-approx-measure}
\ep^{\gamma^2/2} e^{2 h^{\ep}/Q} (dx^2 + dy^2)
\eqe
with $dx^2 + dy^2$ denoting the flat metric on $\BB{C}$.

Alternatively, we can think of approximating the LQG distance function in terms of $h^{\ep}$.   The LQG distance function should be approximated by the distance function induced by the same metric, with a different $\ep$-normalization:
\[
\ep^{\xi Q} e^{2 h^{\ep}/Q} (dx^2 + dy^2)
\]
(One way to see this intuitively is to compare the distance function induced by the latter metric with the Liouville graph distance, defined, e.g., in~\cite{ding-goswami-watabiki}.)

These heuristics motivate the following definition:

\begin{defn}
\label{def-metric-regularized}
On the event that $\cS^{\ep}$ is finite, we define an $\ep$-regularized LQG metric\footnote{We emphasize that the $\ep$-regularized LQG metric is not smooth. Indeed, although $\Delta h^{\ep}$ is well defined as a function, it is typically not continuous; see Lemma~\ref{lem-alpha-law}. We discuss this further in Section~\ref{section-approx-surface}.} on the square $\BB{S}$ with matter central charge $\cc$ as
\[
e^{2 h^{\ep}/Q} (dx^2 + dy^2).
\]
\end{defn}

Importantly, the metric is defined only on the event that $\cS^{\ep}$ is finite. However, it is defined for all values of matter central charge less than $25$, though the induced LQG measure approximation \eqref{eq-approx-measure} is only real-valued for $\cc \leq 1$. This is related to the fact, noted earlier, that $\cS^{\ep}$ is not a.s. finite for $c >1$.

We stress that, though approximating at a quantum scale was critical, our particular square subdivision approach to approximating $h$ at a quantum scale is not canonical.  We could have chosen to define a quantum-regularized version of $h$ by projecting on some other finite-dimensional subspace of $H^1_0(\BB{S})$, and we did not have to choose our domain to be $\BB{S}$.  The reason we chose our particular quantum-scale approximation scheme is that it is easy to describe, and it is closely related to approximations of LQG that have already been described in the literature.  The idea of approximating LQG by subdividing into dyadic squares appears in~\cite{shef-kpz} (in terms of the LQG measure) in the context of the KPZ formula, and more recently in~\cite{ghpr-central-charge} in constructing a proposed model of LQG (as a metric space) for $\cc \in (1,25)$. (We will describe the model in~\cite{ghpr-central-charge} in greater depth in Section~\ref{sec-interpretation2}.)

\subsection{Defining quantum-regularized LQG surfaces and weighting by the zeta-regularized Laplacian determinant}\label{section-approx-surface}

We will implement the regularization idea we described in Section~\ref{section-approx-metric} to construct a quantum-regularized LQG surface with the topology of the sphere. To define the quantum-regularized LQG metric on the sphere, we  essentially just extend the metric we defined in Definition~\ref{def-metric-regularized} to all of $\wh{\BB C}$.  However, we also normalize the metric by adding a constant $C$ in the exponent whose explicit form looks a bit messy.  Before stating the definition of the metric, we explain the motivation for this choice of normalizing constant $C$.

We constructed the metric in Definition~\ref{def-metric-regularized} in such a way that, at least heuristically, all of the squares have quantum (LQG) size about $\ep^{\gamma Q}$.  We now want to essentially extend this definition of metric to the sphere, and we will claim in the next section that, if we set $\ep =1$ and send the number of squares $n$ in $\cS^\ep$ to infinity, this quantum-regularized metric should converge to an LQG surface with the topology of the sphere called the \emph{quantum sphere}.  To that end, we want to normalize our regularized metric so that the sphere has volume approximately equal to some deterministic constant.  We can normalize our regularized metric simply by dividing by the area of the sphere with respect to this approximating metric.  Alternatively, since this area is asymptotically some multiple of $ n$ w.h.p.\ in the $n \rta \infty$ limit (for $\ep = 1$), we can instead normalize our regularized metric by dividing by this multiple of $n$.  More generally, we can normalize by any weighted geometric mean of these two notions of area.  Later in this subsection, we will see that for exactly one choice of weighted geometric mean of these two notions of area, the quantum-regularized LQG metric with this normalization has a particularly nice expression for the zeta-regularized Laplacian determinant.  This is the normalization that we adopt in Definition~\ref{def-lqg-surface-regularized}. Since we ultimately intend to consider random surfaces of the same size, our choice of normalization should not impact the limiting nature of the random surfaces.

We are now ready to state the definition of our quantum-regularized LQG metric on the sphere.

\begin{defn}
\label{def-lqg-surface-regularized}
Let $g_0= e^{2 \sigma_0(z)} dz$ be a (deterministic) smooth metric for which $\sigma \equiv 0$ in a neighborhood of $\BB{S}$, and let $h$ be a GFF on $\wh{\BB C}$ normalized so that $\int h K_0 \ d\Vol_{g_0} = 0$, where $K_0$ is the Gaussian curvature of the sphere with respect to the metric $g_0$. Let $h^{\ep} = h^{\cS^{\ep}}$ be as in Definition~\ref{def-metric-regularized}. 
On the event that the number $n$ of squares in $\cS^{\ep}$ is finite, we define a quantum-regularized LQG surface as the sphere $\wh{\BB{C}} = \BB{C} \cup \{\infty\}$ equipped with the metric 
\[ 
g = e^{2 (h^{\ep}/Q + C)} g_0,
\] 
where the constant $C$ is given by %
\[
C= \frac{1}{4} \log n - \frac{3}{4} \log \Vol_{e^{2 h^{\ep}/Q} g_0}(\wh{\BB C}).
\]
Note that the metric we have defined implicitly depends on the value of the subdivision threshold $\ep > 0$, the number of squares $n$, and the matter central charge $\cc \in (-\infty,25)$ (via $\qq = \qq(\cc)$).  
\end{defn}

The quantum-regularized LQG surface is not smooth, so the short time expansions \eqref{eqn-short-time-closed}, \eqref{eqn-short-time-boundary} are not immediate, although the eigenvalues and the heat kernel are continuously defined (see e.g. \cite{ding2002heat}). Consequently it is not clear to us whether the zeta function can be analytically extended to a neighborhood of zero to define the zeta-regularized Laplacian determinant. We instead make sense of this quantity by the Polyakov-Alvarez formula, an important classical result describing how the zeta-regularized Laplacian determinant transforms under a conformal change of metric.
We note that, in light of Theorem~\ref{thm-loop-det}, the Polyakov-Alvarez formula (in the case of smooth manifolds) can now also be interpreted as a statement about how the regularized mass of Brownian loops changes under a conformal change in metric.

 \begin{proposition}[The Polyakov-Alvarez formula (\cite{polyakov-qg1,alvarez}, see also \cite{osgood-phillips-sarnak})]\label{prop_polyakov_alvarez}
Suppose that $g_0$ and $g = e^{2\sigma} g_0$ are a pair of conformally equivalent smooth metrics on a compact manifold $M$.  Let $K_0$ denote the Gauss curvature with respect to $g_0$, and let $k_0$ denote the geodesic curvature with respect to $g_0$ if $M$ has a boundary.  Then the zeta-regularized determinants of the Laplacian with respect to $g$ and $g_0$ are related by
\begin{align} \label{eq-polyakov-alvarez-closed}
\begin{split} 
&\log {\det}_{\zeta}' \Delta_g = -\frac{1}{12\pi} \int_M |\nabla_{g_0} \sigma|^2 \ d\Vol_{g_0} - \frac{1}{6\pi} \int_M K_0 \sigma \ d \Vol_{g_0} + \log \Vol_{g}(M) \\ &\qquad \qquad \qquad \qquad \qquad \qquad - \log \Vol_{g_0}(M) + \log {\det}'_\zeta \Delta_{g_0}
\end{split}
\end{align}
when $M$ is closed and
\begin{align} \label{eq-polyakov-alvarez-boundary}
\begin{split} 
&\log {\det}_{\zeta} \Delta_g = -\frac{1}{12\pi} \int_M |\nabla_{g_0} \sigma|^2 \ d\Vol_{g_0} - \frac{1}{6\pi} \int_M K_0 \sigma \ d \Vol_{g_0} \\ &\qquad \qquad - \frac{1}{6\pi} \int_{\partial M} k_0 \sigma \ d \Len_{g_0} -\frac{1}{4\pi} \int_{\partial M} \partial_n \sigma \ d \Len_{g_0} + \log {\det}_\zeta \Delta_{g_0},
\end{split}
\end{align}
when $M$ has a boundary.
\end{proposition}

Though $g$ is not smooth in our setting, we may extend the definition of the zeta-regularized determinant of the Laplacian via Proposition~\ref{prop_polyakov_alvarez}.  

The expression~\eqref{eq-polyakov-alvarez-closed} is not exactly scale-invariant---i.e., it does not just depend on $\sigma$ modulo additive constant---because of the second and third terms on the right-hand side of~\eqref{eq-polyakov-alvarez-closed}.  We can add a constant to $\sigma$ so that these two terms vanish.  As we described at the beginning of this subsection, we chose the normalizing constant $C$ in Definition~\ref{def-lqg-surface-regularized}---along with the normalization $(h^\ep, K_0) = 0$ of $h^\ep$---so that these two terms vanish for our regularized LQG metric.  (This is just a simple application of the Gauss-Bonnet theorem.) This means that the zeta-regularized Laplacian determinant of our metric, as defined by Proposition~\ref{prop_polyakov_alvarez}, reduces to a deterministic constant times
\eqb
\exp\left(-\frac{1}{12\pi} \int_{\BB C} |\nabla (2 h^{\ep}/Q)|^2 \ dz\right).
\label{eqn-reg-determinant-laplacian}
\eqe

Having defined a notion of Laplacian determinant in our setting, we obtain the following precise formulation of our main result, which we loosely stated in Section~\ref{sec-intro} as Theorem~\ref{thm-weighting}:

\begin{theorem}
\label{thm-laplacian-reweighting}
Fix a positive integer $n$ and $\cc, \cc' \in \BB{R}$ with $\cc$ and $\cc +\cc'$ less than $25$. With $g$ as in Definition~\ref{def-lqg-surface-regularized}, denote by $\mu_{\cc,n}$ the law of $g$  conditioned on the event $\# \cS^{\ep} = n$. If we weight the law $\mu_{\cc,n}$ by the $(-\cc'/2)$-th power of~\eqref{eqn-reg-determinant-laplacian}, then the reweighted law is exactly $\mu_{\cc_{\op{new}},n}$ with $\cc_{\op{new}} = \cc + \cc'$.
\end{theorem}

\begin{proof}[Proof of Theorem~\ref{thm-laplacian-reweighting}]
Fix a partition $\cS$ of $\BB{S}$ into $n$ dyadic squares.  
By Lemmas~\ref{lem-measurable} and~\ref{lem-alpha-law}, the following are all measurable functions of the set of variables $\{\alpha^{\cS}_S/\qq\}_{S \in \cS}$ that do not depend on the matter central charge $\cc$:
\begin{itemize}
    \item the function $h^{\cS}/Q$,
    \item the event that $\cS^{\ep} = \cS$, and
    \item the additive factor $C$ defined in Definition~\ref{def-lqg-surface-regularized}.
\end{itemize}

The metric $g$ is itself a $\cc$-independent measurable function of $h^{\ep}/Q$ and $C$.  Hence, to prove the theorem, it suffices to consider what happens when we re-weight the law $\nu_{\cc,\cS}$ of the variables $\{\alpha^{\cS}_S/\qq\}_{S \in \cS}$ restricted to the event that  $\cS^{\ep} = \cS$. Specifically, it is enough to show that the law $\nu_{\cc,\cS}$ re-weighted by the $(-\cc'/2)$-th power of~\eqref{eqn-reg-determinant-laplacian} is equal to the law $\nu_{\cc_{\op{new}}, \cS}$ up to a constant independent of the choice of $\cS$ (but which may depend on $n,\cc,\cc'$).

We can write the law $\nu_{\cc, \cS}$ explicitly: up to a normalizing constant depending on $n$ and $\cc$, it is 
\[
\prod_{S \in \cS} e^{-\qq^2 x_S^2/2} \BB{1}_{ (x_S) \in \Lambda_{\cS} },
\]
where $\Lambda_{\cS}$ denotes the set of values of the $x_S$ for which, setting $\alpha_S/\qq = x_S$ for each $S \in \cS$, we obtain 
$\cS^{\ep} = \cS$.  (Recall from the beginning of the proof that $\Lambda_{\cS}$ is independent of matter central charge.)
The re-weighted law is, up to a normalizing constant depending on $n,\cc,\cc'$ but independent of $\cS$,
\[
\prod_{S \in \cS} e^{\frac{\cc'}{12} x_S^2} e^{-\qq^2 x_S^2/2} \BB{1}_{ (x_S)_{S \in \cS} \in \Lambda_{\cS} }
\]
or
\eqb
\prod_S e^{\left(\frac{\cc'}{6 } - \qq^2\right)  x_S^2/2} \BB{1}_{ (x_S)_{S \in \cS} \in \Lambda_{\cS} },
\label{eqn-reweighted-law-qnew}
\eqe
The expression $\qq^2 - \frac{\cc'}{6}$ is equal to the background charge $\qq_{\op{new}}$ corresponding to $\cc_{\op{new}}$.  Hence, the law~\eqref{eqn-reweighted-law-qnew} equals $\nu_{\cc_{\op{new}},\cS}$ up to a normalizing constant depending on $n,\cc,\cc'$ but independent of $\cS$.

\end{proof}

\subsection{Interpreting our result for $\cc < 1$}
\label{sec-interpretation1}

We introduced in Definition~\ref{def-lqg-surface-regularized} an object $(M,g)$, which we interpret as a quantum regularization of an LQG surface. It is known in various $c < 1$ contexts that an LQG surface {\em conditioned} to have large area $A$ (and then rescaled to make the total area some deterministic constant) converges in the $A \to \infty$ to an object called the {\em quantum sphere}, as defined in \cite{wedges, ahs-sphere, sphere-constructions, dkrv-lqg-sphere}.  It is natural therefore to expect that the objects $(M,g)$, conditioned to have a large number of squares, would approximate quantum spheres as well. Although we will not formally prove this, we state it as a claim below and provide some further justification:

\begin{claim}
\label{claim-sphere}
Let $(M,g)$ be as in Definition~\ref{def-lqg-surface-regularized} for $\cc \in (-\infty,1)$.  Then the law of $(M, g)$ for $\ep = 1$ conditioned on  $\# \cS^{1} = n$ converges as $n \rta \infty$ (with $\ep = 1$ fixed) to a quantum sphere with volume some constant deterministic constant---defined, e.g., in~\cite{wedges, ahs-sphere, sphere-constructions, dkrv-lqg-sphere}---in the topology of convergence of quantum surfaces described in~\cite{shef-zipper}.
\end{claim}

As we noted at the beginning of Section~\ref{section-approx-surface}, the metric $g$ is normalized to have asymptotically constant total volume, so that we can expect to obtain an LQG surface with constant total volume in the limit.

\begin{proof}[Justification of Claim~\ref{claim-sphere}]
First, for the reader who may not be familiar with the topology of convergence in the claim, we rephrase the statement of the claim in more concrete terms.
Let $\cC = \R \times S^1$ be the infinite cylinder, and write $\cC_+ = (0,\infty) \times S^1$. We canonically embed both our quantum-regularized LQG surface $M$ and the quantum sphere in $\cC$ as follows:
\begin{itemize}
\item Choose $z \in M$ from the volume measure on $(M, g)$. There exists a conformal map from $\wh \C$ to $\cC$ sending $\infty \mapsto -\infty$ and $z \mapsto +\infty$. This map is unique up to horizontal translation and rotation; we fix the horizontal translation by specifying that the volume of $\cC_+$ under the measure induced by the embedding is $\frac12$, and we choose the rotation on $S^1$ uniformly at random. This gives us a parametrization of $(M, g)$ in $\cC$.  Set $\mu_n$ to be the volume measure on $\cC$ under this embedding.

\item We embed the quantum sphere in $\cC$ by sending its two marked points to $\pm \infty$, choose the rotation on $S^1$ uniformly at random, and fix the horizontal translation so $\cC_+$ has area $\frac12$ under the volume measure $\mu$ on $\cC$ induced by the embedding.
\end{itemize}

The content of Claim~\ref{claim-sphere} is that the measures $\mu_n$ converge in the vague topology in law to the measure $\mu$.%

To justify Claim~\ref{claim-sphere}, we recall from~\cite[Proposition A.11]{wedges} that the quantum sphere with total volume $k$ can be defined by the following limiting procedure. Consider a  zero-boundary GFF $h$ on some smooth domain $D$ conditioned on the event that the induced LQG measure of $D$ is in some large range $[kA,kA(1+\delta)]$ for fixed large $A>0$ and small $\delta>0$.   In other words, we condition the ``normalized'' field $\wh{h} = h - \gamma^{-1} \log{A}$ on the event that the induced LQG measure of $D$ is between $k$ and $k+\delta$.  Then the  quantum sphere is obtained by taking a weak limit, first as $A \rta \infty$ and then $\delta \rta 0$, of the pair $(D,\wh{h})$ in the topology of convergence of quantum surfaces. 
In other words, if we embed $(D, \mu_{\wh h})$ onto a subset of $\cC$ (using the earlier procedure we used to embed $(M, g)$),  then~\cite[Proposition A.11]{wedges} asserts that the measure on $\cC$ induced by this embedding converges  in the vague topology in law to the measure $\mu$ defined above.

There are some differences between the limiting procedures in~\cite[Proposition A.11]{wedges} and Claim~\ref{claim-sphere}.  Perhaps the most technically annoying difference is that conditioning on the number of squares in $\cS^{\ep}$ to be large is not exactly the same as conditioning the LQG measure of the domain to be large.  However, the LQG measure corresponding to the pair $(D,\wh{h})$ for $C$ large is heuristically related to the measure induced by the metric $\wh{g}$ defined in Claim~\ref{claim-sphere} for $n$ large. We believe that, with some technical effort, one could show that the two limiting procedures should yield the same object.
\end{proof}

Thus, we can view Theorem~\ref{thm-laplacian-reweighting} as establishing a heuristic correspondence between two one-parameter family of geometries  indexed by $\cc < 1$: 
\begin{itemize}
    \item 
planar maps with the topology of the sphere sampled with probability proportional to $(\det \Delta)^{-\cc/2}$, and 
\item LQG surfaces (specifically, constant-volume quantum spheres) regularized at a quantum scale, as defined in Definition~\ref{def-metric-regularized}.
\end{itemize}

\subsection{Interpreting our result for $\cc \in (1,25)$}
\label{sec-interpretation2}

We now address the case of matter central charge greater than $1$.\footnote{The following discussion also appears in some form in~\cite[Section 2.2]{ghpr-central-charge}.} From a physics perspective, LQG should also make sense for values of $\cc$ greater than 1, but the geometric behavior of LQG in this regime is still mysterious.  The classical Euclidean-scale regularization~\eqref{def-measure} of the LQG measure does not clearly extend to $\cc > 1$ because the parameter $\gamma$ is no longer real.

In contrast, the quantum-scale regularization approach we defined in Definition~\ref{def-squares}, in which we subdivided the domain into a collection $\cS^{\ep}$ of squares of roughly equal ``quantum size''~\eqref{eqn-mass-def}, is well-defined whenever the background charge $Q$ is real and nonzero, i.e., for all $\cc < 25$.  However, there is a crucial difference in the distribution of this set of squares $\cS^{\ep}$ in the regimes $\cc \leq 1$ and $\cc > 1$.  When $\cc$ is at most $1$, the number of squares in $\cS^{\ep}$ is almost surely finite for any fixed $\ep > 0$.  On the other hand, when $\cc > 1$, the number of squares in $\cS^{\ep}$ is infinite with probability tending to $1$ as $\ep > 0$. (Roughly speaking, this corresponds to the fact that the set of thick points of thickness greater than $Q$ has positive dimension; see~\cite{ghpr-central-charge} for details.)  

As a result, the square subdivision approach suggests \emph{two different natural extensions} of LQG to a one-parameter family of geometries indexed by $\cc < 25$.  If we define a model for LQG for $\cc > 1$ by regularizing at a quantum scale (as we have done by decomposing the domain into the set of squares $\cS^{\ep}$), we have two choices: we can either
\begin{itemize}
\item
condition the number of squares to be finite---which we will call imposing the \emph{finite-volume condition}---or
\item not impose any such conditioning and allow the set of squares $\cS^{\ep}$ to be infinite. 
\end{itemize}
The construction of the quantum-regularized LQG surface in Definition~\ref{def-lqg-surface-regularized} imposes the finite-volume condition; i.e., we condition the number of squares in $\cS^{\ep}$ to be finite.  In contrast, the paper~\cite{ghpr-central-charge} defines a random planar map model for ``LQG with $\cc \in (1,25)$'' using essentially the same square subdivision procedure as in Definition~\ref{def-squares},\footnote{Specifically, they consider the random planar map given by the adjacency graph of squares, with two squares considered adjacent if their edges intersect along a positive-length line segment.} but they do \emph{not} impose the finite-volume condition on the collection of squares in their version of quantum-regularized LQG.
Theorem~\ref{thm-laplacian-reweighting} implies that the one-parameter family of random planar maps weighted by $(\det \Delta)^{-\cc/2}$  should describe the same limiting continuum geometry as quantum-regularized LQG surfaces \emph{with} the finite-volume condition. See Figure~\ref{fig-correspondence}.

\begin{figure}[t!]
\begin{center}
\centering
        \includegraphics[width=\textwidth]{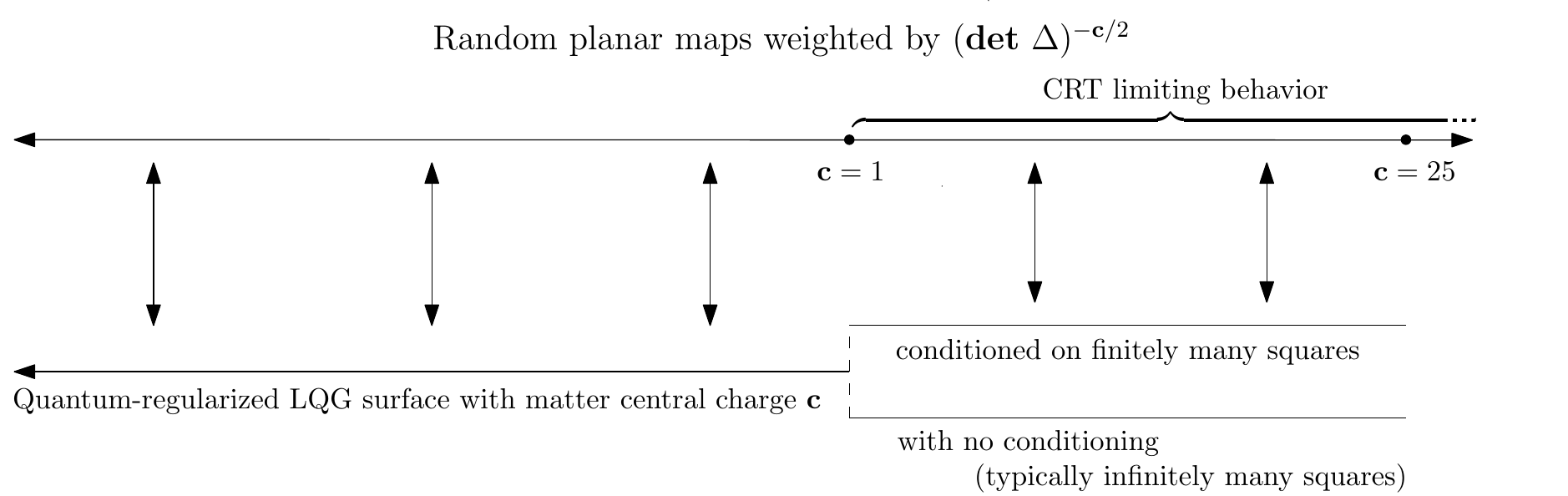}
    \caption{
When $\cc > 1$, we can approximate LQG by subdividing into squares of roughly equal ``quantum size'' in two different ways: by conditioning the number of squares to be finite (the \emph{finite-volume condition}), or by allowing the number of squares to be infinite.  These two approaches are the same for $\cc \leq 1$ since the number of squares is finite almost surely.  The arrows in the figure indicate ``heuristic similarity.''  In light of Theorem~\ref{thm-laplacian-reweighting}, we expect a random planar map weighted by $(\det \Delta)^{-c/2}$ to be heuristically similar to a quantum-regularized LQG surface with matter central charge $c$ (under the finite-volume condition), at least in the sense of having the same limiting behavior. It is an open problem to construct natural random planar map models that are heuristically similar to the infinite-volume quantum-regularized LQG surfaces; see Question~\ref{ques-discrete-c-greater-than-1}.}
\label{fig-correspondence}
\end{center}
\end{figure}

We believe that imposing the finite-volume condition completely changes the geometric behavior of the (conjectural) limiting continuum surface.  As the paper~\cite{ghpr-central-charge} demonstrates, if we do not condition on the number of squares in $\cS^{\ep}$ to be finite, we should obtain in the (conjectural) Gromov-Hausdorff limit an infinite-diameter surface with infinitely many ends and infinite Hausdorff dimension.  See Figure~\ref{fig-two-geometries} for an illustration of this surface. On the other hand, once we condition on the number of squares in $\cS^{\ep}$ to be finite, we expect the geometry to degenerate to a so-called branched polymer; i.e., we expect to obtain a \textit{continuum random tree} (CRT) in the $\ep \rta 0$ limit.  We observe an analogous phenomenon in the case of supercritical Galton-Watson trees.  A supercritical Galton-Watson tree almost surely contains infinitely many edges and has infinitely many ends.  On the other hand, we can define a supercritical Galton-Watson tree conditioned to be finite---in this case, we obtain a tree with finitely many edges that looks very different from its infinite, unconditioned counterpart.

\begin{figure}[ht!]
    \centering
        \includegraphics[width=0.83\textwidth]{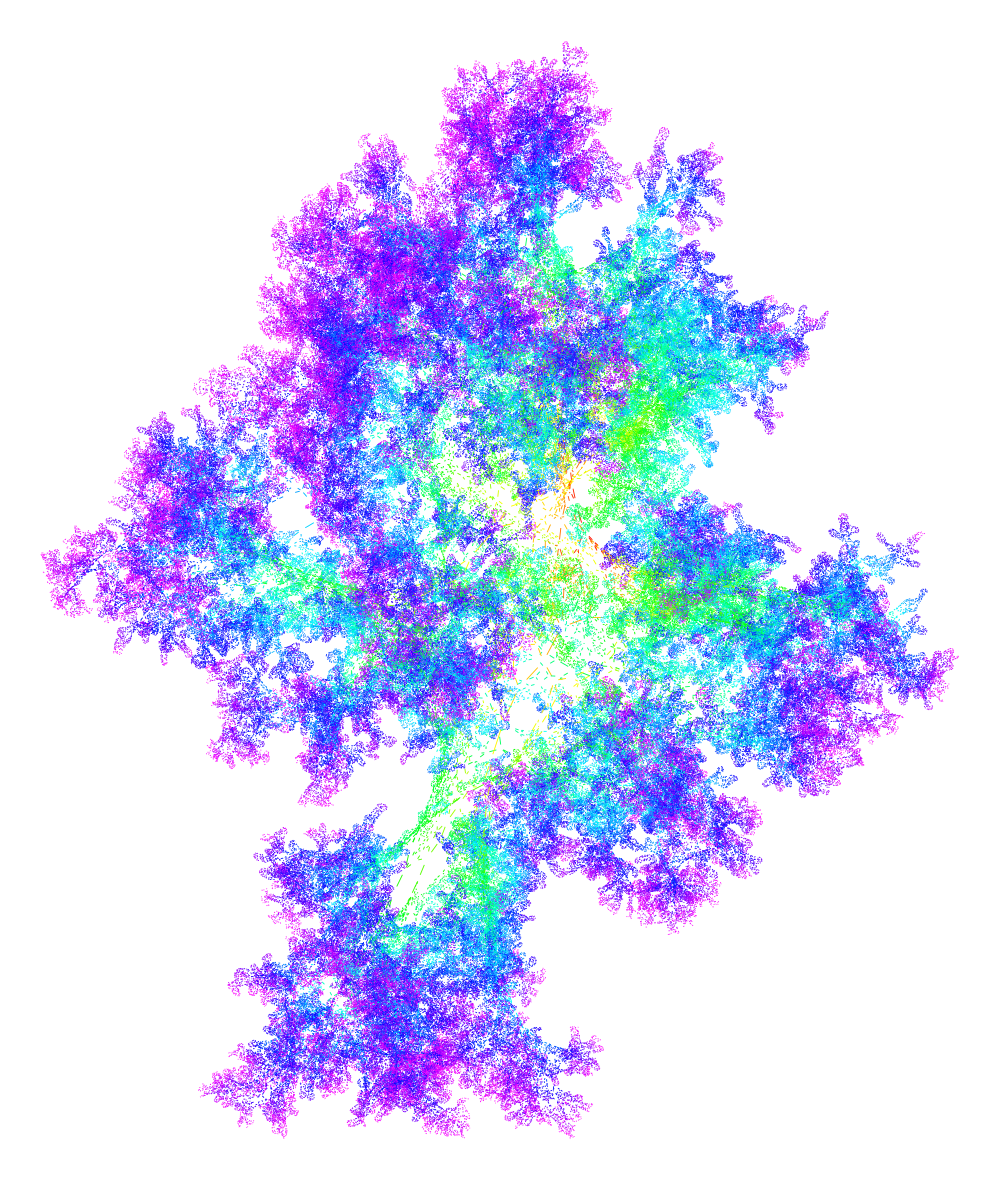}
    \caption[LoF entry]{A three-dimensional force-directed embedding of the radius 50 ball centered at the root in the graph of squares $\cS^{\ep}$ for $c=24$ and $\ep = 10^{-5}$.  To render the simulation, we took the field to be a zero-boundary GFF, and we terminated the square subdivision algorithm at squares smaller than some fixed Euclidean size.  The edges are colored in a rainbow scheme according to their distances from a fixed square, with the purple edges the edges farthest from the root.  The vast majority of edges are purple, as we expect since the volume growth of a metric ball in this random planar map is almost surely super-polynomial in the radius---an indication that the conjectural limiting continuum geometry has infinite Hausdorff dimension almost surely~\cite{ghpr-central-charge}.  (We note that the fact that most edges are purple may not be obvious from the simulation because the regions of purple edges have a relatively high density of edges.  This is similar to what one sees in simulations of a supercritical Galton-Watson tree.)}
    \label{fig-two-geometries}
\end{figure}

This dichotomy that we observe for $\cc > 1$ also explains why many physicists~\cite{cates-branched-polymer,david-c>1-barrier,adjt-c-ge1,ckr-c-ge1,bh-c-ge1-matrix,dfj-critical-behavior,bj-potts-sim,adf-critical-dimensions} who investigated the geometric behavior LQG for $\cc > 1$ observed, through numerical simulations and heuristics, that LQG in this regime seems to behave like a branched polymer.  This is an unsatisfying conclusion from a conformal field theory perspective, since this geometry does not depend on $\cc$ and cannot be related to conformal field theory models (since trees have no conformal structure).  In light of Theorem~\ref{thm-laplacian-reweighting}, it is  not surprising that they reached this conclusion, since our theorem heuristically suggests that weighting random planar maps by $(\det \Delta)^{-\cc/2}$ describes the same limiting geometry that we get when we impose  the finite-volume condition on our quantum-regularized LQG surfaces.  In other words, when $\cc >1$, the approximation of LQG by random planar maps weighted by $(\det \Delta)^{-\cc/2}$  is \emph{not} a good heuristic for the geometry of unconditioned quantum-regularized LQG surfaces.

\section{Open questions} \label{sec-open}

The perspective we describe in this paper suggests many interesting open questions to explore.

First, as noted in Section~\ref{sec-intro}, a natural approach to making the heuristic of ``Definition''~\ref{def-original} precise is to define LQG as a limit (in some sense) of classes of random planar maps sampled with probability proportional to the the $(-\cc/2)$-th power of the determinant of its discrete Laplacian.  The missing element we require to implement this approach is an affirmative answer to the following open question:

\begin{ques}
\label{ques-convergence}
Consider a random planar map---or, alternatively, a $p$-angulation for some $p$---with the topology of the sphere with exactly $n$ edges, sampled with probability  proportional to the $(-\cc/2)$-th power of the determinant of its discrete Laplacian.  Can one prove {\em in any sense} that this random planar map converges in the suitably rescaled $n \rta \infty$ limit to an LQG sphere with matter central charge $\cc$? Is there an analogue of this statement that holds for random planar maps with the topology of the disk?
\end{ques}

The convergence in Question~\ref{ques-convergence} can be interpreted in several different ways: for instance, one can consider convergence in the Gromov-Hausdorff topology, weak convergence of the random measure induced by some ``discrete conformal embedding'' (see, e.g.,~\cite{hs-cardy-embedding}), or some topology that encodes spanning tree or loop behavior. It would be interesting to prove convergence in any of these topologies.

Now, suppose we had a convergence result for any of the random planar map models considered in Question~\ref{ques-convergence} with the topology of the disk.  We could then ask about the limiting behavior of discrete analogues of Brownian loop soups on these models, such as those we described in Section~\ref{sec-discrete}. We know from~\cite{shef-werner-cle} that the set of outer boundaries of a $c$-intensity Brownian loop soup for $c \in (0,1)$ has the law of a so-called \emph{conformal loop ensemble} (defined, e.g., in~\cite{shef-cle}).  We therefore ask whether the outer boundaries of the discrete loop soups converge to a set of self-avoiding loops with this same law.

\begin{ques}
\label{ques-cle}
Suppose that we can give an affirmative answer to Question~\ref{ques-convergence} for some random planar map model with the topology of the disk. Can we show that, if we decorate this model by some discrete analogue of the $c$-intensity Brownian loop soup for $c \in (0,1)$, then the boundaries of the loop soup clusters converge in some sense to the conformal loop ensemble corresponding to that value of $c$?
\end{ques}

We can also consider whether we can justify ``Definition''~\ref{def-original}  directly in the continuum without having to regularize the LQG surface, as we have done in this paper. One can define a loop measure on an LQG surface in terms of the Brownian motion associated to the heuristic LQG metric, which has been rigorously defined and is called \emph{Liouville Brownian motion} (see, e.g.,~\cite{grv-lbm, berestycki-lbm, rhodes-vargas-criticality}).  
We can then consider the law of an LQG surface weighted by a regularized mass of ``Liouville Brownian loops''---i.e., the measure of Brownian loops regularized as in Theorem~\ref{thm-loop-det} or Proposition~\ref{prop-decay}, but with its quadratic variation measured in the Liouville sense.

\begin{ques}
\label{ques-lqg-loop-soup}
Can we justify ``Definition''~\ref{def-original} by directly weighting a LQG surface (e.g., an LQG sphere) by a regularized mass of ``Liouville Brownian loops'' and taking a limit?
\end{ques}

Perhaps an easier task would be to strengthen the result of Theorem~\ref{thm-laplacian-reweighting} by proving Claim~\ref{claim-sphere}.

\begin{ques}
Can we prove Claim~\ref{claim-sphere}?
\end{ques}

A proof of Claim~\ref{claim-sphere} would strengthen the story of Theorem~\ref{thm-laplacian-reweighting} illustrated by Figure~\ref{RLQGchart} by strengthening the upward arrows in the figure (for $\cc \leq 1$).  As noted in Section~\ref{sec-interpretation1}, one must check that conditioning on the square subdivision consisting of a large number of squares is similar to conditioning on the unit square having large LQG area in the limit.

We can also try to strengthen the story of Theorem~\ref{thm-laplacian-reweighting} to obtain a result about weighting by powers of the Brownian loop soup partition function. As we discussed at the end of Section~\ref{sec-relate}, we need to weight by the \emph{(area and boundary length)-adjusted} mass of Brownian loops with quadratic variation between $\epsilon$ and $C$; i.e., we need to discard the leading order terms of the expansion in Theorem~\ref{thm-loop-det}. Theorem~\ref{thm-laplacian-reweighting} already implies that the Radon-Nikodym derivative given by this adjusted truncated mass of Brownian loops converges pointwise, but we do not have enough uniform control to obtain convergence in law; see the discussion at the end of Section~\ref{sec-relate}.

\begin{ques}
\label{ques-loop-version}
Can we show that weighting the law $\mu_{\cc,n}$ of Theorem~\ref{thm-laplacian-reweighting} by $\cc'$ times the (area and boundary length)-adjusted mass of Brownian loops on the surface with quadratic variation between $\epsilon$ and $C$ converges (as $\ep \rta 0$ and $C \rta \infty$) to the law $\mu_{\cc + \cc',n}$?
\end{ques}

Another interesting avenue to explore is a  rigorous Laplacian determinant reweighting story for SLE, given the connection between the partition function of SLE and the zeta-regularized determinant of the Laplacian  (see, e.g.,~\cite[Definition 3.1]{dubedat-coupling}).  For instance, we pose the following question.

\begin{ques}
Suppose that $\eta$ is a random curve in some discrete setting---either a fixed lattice or a random planar map---whose scaling limit is $\SLE_\kappa$ for some $\kappa$.  Can we show that, if we \emph{weight} the law of $\eta$ by $e^{\lambda m(\eta)}$ for $m(\eta)$ some multiple of the (regularized) mass of Brownian loops intersecting $\eta$, then the scaling limit of $\eta$ with this weighted law is $\SLE$ with a different value of $\kappa$?
\end{ques}
A more concrete version of this question is given in \cite[Section 6]{kozdron2007configurational}, where they define the ``$\lambda$-SAW'', a one-parameter family of measures on self-avoiding walks in $\Z^2$: roughly speaking, $\eta$ is a uniform self-avoiding walk, and the Brownian loops are ``random walk loops''. The $\lambda$-SAW is conjectured to converge in the scaling limit to $\mathrm{SLE}_\kappa$ for some explicit $\kappa = \kappa(\lambda)$; this has only been proved for $\lambda = 1$ (corresponding to the convergence of loop-erased random walks to $\mathrm{SLE}_2$).

Next, we turn to the case of LQG with matter central charge $\cc \in (1,25)$. The recent paper~\cite{ghpr-central-charge} proposed a random planar map model for LQG in this phase that describes an infinite volume surface with a.s. infinitely many ends.  In Section~\ref{sec-interpretation2}, we explained why this model differs from predictions in the physics literature that LQG in this phase should behave like the CRT.  In our explanation, we claimed that our model that we considered in Theorem~\ref{thm-laplacian-reweighting} should converge to a CRT, since we conditioned the number of squares in the associated square subdivision to be finite.  It would be interesting to prove some version of this claim:

\begin{ques}
\label{ques-crt}
Can we prove that, if $\cc \in (1,25)$, then the regularized LQG surfaces of Theorem~\ref{thm-laplacian-reweighting}, in which we impose the finite-volume condition, converge in law as $\ep \rta 0$ to a CRT? Alternatively, can we prove this convergence for the adjacency graph of the squares in Definition~\ref{def-squares}---with two squares adjacent if they intersect along a non-trivial connected line segment---in which we condition the set of squares to be finite? (The latter model more closely resembles that proposed in~\cite{ghpr-central-charge}, but with the crucial difference that the set of squares is conditioned to be finite.)
\end{ques}

Next, we would like to understand how the candidate picture of LQG described by the model in~\cite{ghpr-central-charge} relates to sampling according to powers of $\det \Delta_{\op{graph}}$. If simply sampling a random planar map with probability proportional to $(\det \Delta_{\op{graph}})^{-\cc/2}$ should ``degenerate'' to a CRT in the limit, then perhaps it is possible to obtain a more interesting limiting continuum geometry by imposing some conditioning on the random planar map to enforce nondegeneracy. One possible candidate for such a model is a random planar map metric ball of radius $r$, conditioned on the ratio of its boundary length to its number of edges lying in some fixed interval independent of $r$.

\begin{ques}
\label{ques-discrete-c-greater-than-1}
Can one identify a random planar map model that is natural in the discrete setting (not one, like that considered in~\cite{ghpr-central-charge}, that is defined in terms of the continuum free field) that converges to the conjectured infinite-volume continuum $\cc \in (1,25)$ surfaces described in~\cite{ghpr-central-charge}?   This random planar map model should also make sense for $\cc \leq 1$ and converge to LQG with matter central charge $\cc$ in that range.
\end{ques}

Finally, we would like to extend the result of Theorem~\ref{thm-laplacian-reweighting} to higher-genus LQG surfaces.

\begin{ques}
Can the result of Theorem~\ref{thm-laplacian-reweighting} for surfaces with the topology of the sphere be extended to LQG surfaces of higher genus?
\end{ques}

The hard part is understanding what the {\em a priori} law of the conformal class should be for each $\cc$, and then showing that weighting by the determinant of the Laplacian changes that law in the appropriate way. We note that recently, Guillarmou, Rhodes and Vargas \cite{grv-higher-genus} gave a rigorous construction of LQG on surfaces of higher genus corresponding to Polyakov's formulation \cite{polyakov-qg1}, including a description of the law of the conformal class in terms of moments of certain Gaussian multiplicative chaos measures \cite[Theorem 5.1]{grv-higher-genus}.

\bibliographystyle{hmralphaabbrv}
\addcontentsline{toc}{section}{References}
\bibliography{cibib,bibmore}

\end{document}